\documentclass[12pt]{amsart}
\usepackage[latin1]{inputenc}
\usepackage{amsbsy}
\usepackage{amscd} 
\usepackage{amsfonts}
\usepackage{amsgen} 
\usepackage{amsmath}
\usepackage{amsopn} 
\usepackage{amssymb}
\usepackage{amstext}
\usepackage{amsthm} 
\usepackage{amsxtra}
\usepackage{array}
\usepackage[all]{xy}
\usepackage{epsfig}
\usepackage{dsfont}
\usepackage{float}
\usepackage{hhline}
\usepackage{longtable}
\usepackage{mathtools}
\usepackage{paralist}
\usepackage{rotating}
\usepackage{tikz}

\usepackage{calc}
\newcounter{PartitionDepth}
\newcounter{PartitionLength}
\newcommand{\parti}[2]{
 \begin{picture}(#2,#1)
 \setcounter{PartitionDepth}{-1-#1}
 \put(#2,\thePartitionDepth){\line(0,1){#1}}
 \end{picture}}
\newcommand{\partii}[3]{
 \begin{picture}(#3,#1)
 \setcounter{PartitionLength}{#3-#2}
 \setcounter{PartitionDepth}{-1-#1}
 \put(#2,\thePartitionDepth){\line(0,1){#1}}     
 \put(#3,\thePartitionDepth){\line(0,1){#1}}
 \put(#2,\thePartitionDepth){\line(1,0){\thePartitionLength}}
 \end{picture}}

\newcommand{\upparti}[2]{
 \begin{picture}(#2,#1)
 \setcounter{PartitionDepth}{#1}
 \put(#2,0){\line(0,1){#1}}
 \end{picture}}
\newcommand{\uppartii}[3]{
 \begin{picture}(#3,#1)
 \setcounter{PartitionLength}{#3-#2}
 \setcounter{PartitionDepth}{#1}
 \put(#2,0){\line(0,1){#1}}     
 \put(#3,0){\line(0,1){#1}}
 \put(#2,\thePartitionDepth){\line(1,0){\thePartitionLength}}
 \end{picture}}

\newsavebox{\boxpaarpart}
   \savebox{\boxpaarpart}
   { \begin{picture}(0.7,0.5)
     \put(0,0){\line(0,1){0.4}}
     \put(0.5,0){\line(0,1){0.4}}
     \put(0,0.4){\line(1,0){0.5}}
     \end{picture}}
\newsavebox{\boxbaarpart}
   \savebox{\boxbaarpart}
   { \begin{picture}(0.7,0.5)
     \put(0,0){\line(0,1){0.4}}
     \put(0.5,0){\line(0,1){0.4}}
     \put(0,0){\line(1,0){0.5}}
     \end{picture}}
\newsavebox{\boxdreipart}
   \savebox{\boxdreipart}
   { \begin{picture}(1,0.5)
     \put(0,0){\line(0,1){0.4}}
     \put(0.4,0){\line(0,1){0.4}}
     \put(0.8,0){\line(0,1){0.4}}
     \put(0,0.4){\line(1,0){0.8}}
     \end{picture}}
\newsavebox{\boxvierpart}
   \savebox{\boxvierpart}
   { \begin{picture}(1.4,0.5)
     \put(0,0){\line(0,1){0.4}}
     \put(0.4,0){\line(0,1){0.4}}
     \put(0.8,0){\line(0,1){0.4}}
     \put(1.2,0){\line(0,1){0.4}}
     \put(0,0.4){\line(1,0){1.2}}
     \end{picture}}
\newsavebox{\boxvierpartrot}
   \savebox{\boxvierpartrot}
   { \begin{picture}(0.5,0.8)
     \put(0,-0.2){\line(0,1){0.3}}
     \put(0.4,-0.2){\line(0,1){0.3}}
     \put(0,0.1){\line(1,0){0.4}}
     \put(0,0.4){\line(0,1){0.3}}
     \put(0.4,0.4){\line(0,1){0.3}}
     \put(0,0.4){\line(1,0){0.4}}
     \put(0.2,0.1){\line(0,1){0.3}}
     \end{picture}}
\newsavebox{\boxvierpartrotdrei}
   \savebox{\boxvierpartrotdrei}
   { \begin{picture}(1,0.8)
     \put(0,-0.2){\line(0,1){0.3}}
     \put(0.4,-0.2){\line(0,1){0.8}}
     \put(0.8,-0.2){\line(0,1){0.3}}
     \put(0,0.1){\line(1,0){0.8}}
     \end{picture}}
\newsavebox{\boxcrosspart}
   \savebox{\boxcrosspart}
   { \begin{picture}(0.5,0.8)
     \put(0,-0.2){\line(1,2){0.4}}
     \put(0.4,-0.2){\line(-1,2){0.4}}
     \end{picture}}
\newsavebox{\boxhalflibpart}
   \savebox{\boxhalflibpart}
   { \begin{picture}(1,0.8)
     \put(0,-0.2){\line(1,1){0.8}}
     \put(0.8,-0.2){\line(-1,1){0.8}}
     \put(0.4,-0.2){\line(0,1){0.8}}
     \end{picture}}
\newsavebox{\boxpositioner} 
   \savebox{\boxpositioner}
   { \begin{picture}(1.4,0.5)
     \put(0,0){\line(0,1){0.3}}
     \put(0.4,0){\line(0,1){0.6}}
     \put(0.8,0){\line(0,1){0.3}}
     \put(1.2,0){\line(0,1){0.6}}
     \put(0.4,0.6){\line(1,0){0.8}}
     \end{picture}}
\newsavebox{\boxfatcross} 
   \savebox{\boxfatcross}
   { \begin{picture}(1.4,1.3)
     \put(0,-0.2){\line(0,1){0.3}}
     \put(0.4,-0.2){\line(0,1){0.3}}
     \put(0.8,-0.2){\line(0,1){0.3}}
     \put(1.2,-0.2){\line(0,1){0.3}}
     \put(0,0.1){\line(1,0){0.4}}
     \put(0.8,0.1){\line(1,0){0.4}}
     \put(0,0.9){\line(0,1){0.3}}
     \put(0.4,0.9){\line(0,1){0.3}}
     \put(0.8,0.9){\line(0,1){0.3}}
     \put(1.2,0.9){\line(0,1){0.3}}
     \put(0,0.9){\line(1,0){0.4}}
     \put(0.8,0.9){\line(1,0){0.4}}
     \put(0.2,0.1){\line(1,1){0.8}}
     \put(0.2,0.9){\line(1,-1){0.8}}
     \end{picture}}
\newsavebox{\boxprimarypart} 
   \savebox{\boxprimarypart}
   { \begin{picture}(1,1.3)
     \put(0,-0.2){\line(0,1){0.3}}
     \put(0.4,-0.2){\line(0,1){0.3}}
     \put(0.8,-0.2){\line(0,1){0.3}}
     \put(0.4,0.1){\line(1,0){0.4}}
     \put(0,0.9){\line(0,1){0.3}}
     \put(0.4,0.9){\line(0,1){0.3}}
     \put(0.8,0.9){\line(0,1){0.3}}
     \put(0,0.9){\line(1,0){0.4}}
     \put(0,0.1){\line(1,1){0.8}}
     \put(0.2,0.9){\line(1,-2){0.4}}
     \end{picture}}

\newsavebox{\boxidpartww}
   \savebox{\boxidpartww}
   { \begin{picture}(0.5,1.2)
     \put(0.2,0.15){\line(0,1){0.7}}
     \put(0,-0.2){$\circ$}
     \put(0,0.8){$\circ$}
     \end{picture}}
\newsavebox{\boxidpartbw}
   \savebox{\boxidpartbw}
   { \begin{picture}(0.5,1.2)
     \put(0.2,0.15){\line(0,1){0.7}}
     \put(0,-0.2){$\circ$}
     \put(0,0.8){$\bullet$}
     \end{picture}}
\newsavebox{\boxidpartwb}
   \savebox{\boxidpartwb}
   { \begin{picture}(0.5,1.2)
     \put(0.2,0.15){\line(0,1){0.7}}
     \put(0,-0.2){$\bullet$}
     \put(0,0.8){$\circ$}
     \end{picture}}
\newsavebox{\boxidpartbb}
   \savebox{\boxidpartbb}
   { \begin{picture}(0.5,1.2)
     \put(0.2,0.15){\line(0,1){0.7}}
     \put(0,-0.2){$\bullet$}
     \put(0,0.8){$\bullet$}
     \end{picture}}

\newsavebox{\boxidpartsingletonbb}
   \savebox{\boxidpartsingletonbb}
   { \begin{picture}(0.5,1.2)
     \put(0.2,0.15){\line(0,1){0.2}}
     \put(0.2,0.65){\line(0,1){0.2}}
     \put(0,-0.2){$\bullet$}
     \put(0,0.8){$\bullet$}
     \end{picture}}
\newsavebox{\boxidpartsingletonww}
   \savebox{\boxidpartsingletonww}
   { \begin{picture}(0.5,1.2)
     \put(0.2,0.15){\line(0,1){0.2}}
     \put(0.2,0.65){\line(0,1){0.2}}
     \put(0,-0.2){$\circ$}
     \put(0,0.8){$\circ$}
     \end{picture}}
     
\newsavebox{\boxpaarpartbb}
   \savebox{\boxpaarpartbb}
   { \begin{picture}(1,1)
     \put(0.2,0.2){\line(0,1){0.4}}
     \put(0.7,0.2){\line(0,1){0.4}}
     \put(0.2,0.6){\line(1,0){0.5}}
     \put(0.5,-0.2){$\bullet$}
     \put(0,-0.2){$\bullet$}
     \end{picture}}
\newsavebox{\boxpaarpartww}
   \savebox{\boxpaarpartww}
   { \begin{picture}(1,1)
     \put(0.2,0.2){\line(0,1){0.4}}
     \put(0.7,0.2){\line(0,1){0.4}}
     \put(0.2,0.6){\line(1,0){0.5}}
     \put(0.5,-0.2){$\circ$}
     \put(0,-0.2){$\circ$}
     \end{picture}}
\newsavebox{\boxpaarpartbw}
   \savebox{\boxpaarpartbw}
   { \begin{picture}(1,1)
     \put(0.2,0.2){\line(0,1){0.4}}
     \put(0.7,0.2){\line(0,1){0.4}}
     \put(0.2,0.6){\line(1,0){0.5}}
     \put(0.5,-0.2){$\circ$}
     \put(0,-0.2){$\bullet$}
     \end{picture}}
\newsavebox{\boxpaarpartwb}
   \savebox{\boxpaarpartwb}
   { \begin{picture}(1,1)
     \put(0.2,0.2){\line(0,1){0.4}}
     \put(0.7,0.2){\line(0,1){0.4}}
     \put(0.2,0.6){\line(1,0){0.5}}
     \put(0.5,-0.2){$\bullet$}
     \put(0,-0.2){$\circ$}
     \end{picture}}
\newsavebox{\boxbaarpartbb}
   \savebox{\boxbaarpartbb}
   { \begin{picture}(1,1)
     \put(0.2,-0.1){\line(0,1){0.4}}
     \put(0.7,-0.1){\line(0,1){0.4}}
     \put(0.2,-0.1){\line(1,0){0.5}}
     \put(0.5,0.3){$\bullet$}
     \put(0,0.3){$\bullet$}
     \end{picture}}
\newsavebox{\boxbaarpartww}
   \savebox{\boxbaarpartww}
   { \begin{picture}(1,1)
     \put(0.2,-0.1){\line(0,1){0.4}}
     \put(0.7,-0.1){\line(0,1){0.4}}
     \put(0.2,-0.1){\line(1,0){0.5}}
     \put(0.5,0.3){$\circ$}
     \put(0,0.3){$\circ$}
     \end{picture}}
\newsavebox{\boxbaarpartbw}
   \savebox{\boxbaarpartbw}
   { \begin{picture}(1,1)
     \put(0.2,-0.1){\line(0,1){0.4}}
     \put(0.7,-0.1){\line(0,1){0.4}}
     \put(0.2,-0.1){\line(1,0){0.5}}
     \put(0.5,0.3){$\circ$}
     \put(0,0.3){$\bullet$}
     \end{picture}}
\newsavebox{\boxbaarpartwb}
   \savebox{\boxbaarpartwb}
   { \begin{picture}(1,1)
     \put(0.2,-0.1){\line(0,1){0.4}}
     \put(0.7,-0.1){\line(0,1){0.4}}
     \put(0.2,-0.1){\line(1,0){0.5}}
     \put(0.5,0.3){$\bullet$}
     \put(0,0.3){$\circ$}
     \end{picture}}
\newsavebox{\boxpaarbaarpartwwww}
   \savebox{\boxpaarbaarpartwwww}
   { \begin{picture}(1,1.2)
     \put(0.2,0.15){\line(0,1){0.2}}
     \put(0.2,0.65){\line(0,1){0.2}}
     \put(0.2,0.65){\line(1,0){0.5}}
     \put(0.2,0.35){\line(1,0){0.5}}
     \put(0.7,0.15){\line(0,1){0.2}}
     \put(0.7,0.65){\line(0,1){0.2}}
     \put(0,0.8){$\circ$}
     \put(0.5,0.8){$\circ$}
     \put(0,-0.2){$\circ$}
     \put(0.5,-0.2){$\circ$}
     \end{picture}}     

\newsavebox{\boxsingletonw}
   \savebox{\boxsingletonw}
   { \begin{picture}(0.5, 1)
     \put(0,0.3){$\uparrow$}
     \put(0,-0.2){$\circ$}
     \end{picture}}
\newsavebox{\boxsingletonb}
   \savebox{\boxsingletonb}
   { \begin{picture}(0.5, 1)
     \put(0,0.3){$\uparrow$}
     \put(0,-0.2){$\bullet$}
     \end{picture}}
\newsavebox{\boxdownsingletonw}
   \savebox{\boxdownsingletonw}
   { \begin{picture}(0.5, 1)
     \put(0,0){$\downarrow$}
     \put(0,0.6){$\circ$}
     \end{picture}}
\newsavebox{\boxdownsingletonb}
   \savebox{\boxdownsingletonb}
   { \begin{picture}(0.5, 1)
     \put(0,0){$\downarrow$}
     \put(0,0.6){$\bullet$}
     \end{picture}}

\newsavebox{\boxvierpartwbwb}
   \savebox{\boxvierpartwbwb}
   { \begin{picture}(1.7,1)
     \put(0.2,0.2){\line(0,1){0.4}}
     \put(0.6,0.2){\line(0,1){0.4}}
     \put(1,0.2){\line(0,1){0.4}}
     \put(1.4,0.2){\line(0,1){0.4}}
     \put(0.2,0.6){\line(1,0){1.2}}
     \put(0,-0.2){$\circ$}
     \put(0.4,-0.2){$\bullet$}
     \put(0.8,-0.2){$\circ$}
     \put(1.2,-0.2){$\bullet$}
     \end{picture}}
\newsavebox{\boxvierpartbwbw}
   \savebox{\boxvierpartbwbw}
   { \begin{picture}(1.7,1)
     \put(0.2,0.2){\line(0,1){0.4}}
     \put(0.6,0.2){\line(0,1){0.4}}
     \put(1,0.2){\line(0,1){0.4}}
     \put(1.4,0.2){\line(0,1){0.4}}
     \put(0.2,0.6){\line(1,0){1.2}}
     \put(0,-0.2){$\bullet$}
     \put(0.4,-0.2){$\circ$}
     \put(0.8,-0.2){$\bullet$}
     \put(1.2,-0.2){$\circ$}
     \end{picture}}
\newsavebox{\boxvierpartwwbb}
   \savebox{\boxvierpartwwbb}
   { \begin{picture}(1.7,1)
     \put(0.2,0.2){\line(0,1){0.4}}
     \put(0.6,0.2){\line(0,1){0.4}}
     \put(1,0.2){\line(0,1){0.4}}
     \put(1.4,0.2){\line(0,1){0.4}}
     \put(0.2,0.6){\line(1,0){1.2}}
     \put(0,-0.2){$\circ$}
     \put(0.4,-0.2){$\circ$}
     \put(0.8,-0.2){$\bullet$}
     \put(1.2,-0.2){$\bullet$}
     \end{picture}}
\newsavebox{\boxvierpartwbbw}
   \savebox{\boxvierpartwbbw}
   { \begin{picture}(1.7,1)
     \put(0.2,0.2){\line(0,1){0.4}}
     \put(0.6,0.2){\line(0,1){0.4}}
     \put(1,0.2){\line(0,1){0.4}}
     \put(1.4,0.2){\line(0,1){0.4}}
     \put(0.2,0.6){\line(1,0){1.2}}
     \put(0,-0.2){$\circ$}
     \put(0.4,-0.2){$\bullet$}
     \put(0.8,-0.2){$\bullet$}
     \put(1.2,-0.2){$\circ$}
     \end{picture}}
\newsavebox{\boxvierpartbwwb}
   \savebox{\boxvierpartbwwb}
   { \begin{picture}(1.7,1)
     \put(0.2,0.2){\line(0,1){0.4}}
     \put(0.6,0.2){\line(0,1){0.4}}
     \put(1,0.2){\line(0,1){0.4}}
     \put(1.4,0.2){\line(0,1){0.4}}
     \put(0.2,0.6){\line(1,0){1.2}}
     \put(0,-0.2){$\bullet$}
     \put(0.4,-0.2){$\circ$}
     \put(0.8,-0.2){$\circ$}
     \put(1.2,-0.2){$\bullet$}
     \end{picture}}
\newsavebox{\boxvierpartrotwbwb}
   \savebox{\boxvierpartrotwbwb}
   { \begin{picture}(1,1.2)
     \put(0.2,0.15){\line(0,1){0.2}}
     \put(0.2,0.65){\line(0,1){0.2}}
     \put(0.2,0.65){\line(1,0){0.5}}
     \put(0.2,0.35){\line(1,0){0.5}}
     \put(0.45,0.35){\line(0,1){0.3}}
     \put(0.7,0.15){\line(0,1){0.2}}
     \put(0.7,0.65){\line(0,1){0.2}}
     \put(0,0.8){$\circ$}
     \put(0.5,0.8){$\bullet$}
     \put(0,-0.2){$\circ$}
     \put(0.5,-0.2){$\bullet$}
     \end{picture}}
\newsavebox{\boxvierpartrotbwwb}
   \savebox{\boxvierpartrotbwwb}
   { \begin{picture}(1,1.2)
     \put(0.2,0.15){\line(0,1){0.2}}
     \put(0.2,0.65){\line(0,1){0.2}}
     \put(0.2,0.65){\line(1,0){0.5}}
     \put(0.2,0.35){\line(1,0){0.5}}
     \put(0.45,0.35){\line(0,1){0.3}}
     \put(0.7,0.15){\line(0,1){0.2}}
     \put(0.7,0.65){\line(0,1){0.2}}
     \put(0,0.8){$\bullet$}
     \put(0.5,0.8){$\circ$}
     \put(0,-0.2){$\circ$}
     \put(0.5,-0.2){$\bullet$}
     \end{picture}}
\newsavebox{\boxvierpartrotbwbw}
   \savebox{\boxvierpartrotbwbw}
   { \begin{picture}(1,1.2)
     \put(0.2,0.15){\line(0,1){0.2}}
     \put(0.2,0.65){\line(0,1){0.2}}
     \put(0.2,0.65){\line(1,0){0.5}}
     \put(0.2,0.35){\line(1,0){0.5}}
     \put(0.45,0.35){\line(0,1){0.3}}
     \put(0.7,0.15){\line(0,1){0.2}}
     \put(0.7,0.65){\line(0,1){0.2}}
     \put(0,0.8){$\bullet$}
     \put(0.5,0.8){$\circ$}
     \put(0,-0.2){$\bullet$}
     \put(0.5,-0.2){$\circ$}
     \end{picture}}
\newsavebox{\boxvierpartrotwwww}
   \savebox{\boxvierpartrotwwww}
   { \begin{picture}(1,1.2)
     \put(0.2,0.15){\line(0,1){0.2}}
     \put(0.2,0.65){\line(0,1){0.2}}
     \put(0.2,0.65){\line(1,0){0.5}}
     \put(0.2,0.35){\line(1,0){0.5}}
     \put(0.45,0.35){\line(0,1){0.3}}
     \put(0.7,0.15){\line(0,1){0.2}}
     \put(0.7,0.65){\line(0,1){0.2}}
     \put(0,0.8){$\circ$}
     \put(0.5,0.8){$\circ$}
     \put(0,-0.2){$\circ$}
     \put(0.5,-0.2){$\circ$}
     \end{picture}}
\newsavebox{\boxvierpartrotbbbb}
   \savebox{\boxvierpartrotbbbb}
   { \begin{picture}(1,1.2)
     \put(0.2,0.15){\line(0,1){0.2}}
     \put(0.2,0.65){\line(0,1){0.2}}
     \put(0.2,0.65){\line(1,0){0.5}}
     \put(0.2,0.35){\line(1,0){0.5}}
     \put(0.45,0.35){\line(0,1){0.3}}
     \put(0.7,0.15){\line(0,1){0.2}}
     \put(0.7,0.65){\line(0,1){0.2}}
     \put(0,0.8){$\bullet$}
     \put(0.5,0.8){$\bullet$}
     \put(0,-0.2){$\bullet$}
     \put(0.5,-0.2){$\bullet$}
     \end{picture}}

\newsavebox{\boxdreipartwww}
   \savebox{\boxdreipartwww}
   { \begin{picture}(1.2,1)
     \put(0.2,0.2){\line(0,1){0.4}}
     \put(0.6,0.2){\line(0,1){0.4}}
     \put(1,0.2){\line(0,1){0.4}}
     \put(0.2,0.6){\line(1,0){0.8}}
     \put(0,-0.2){$\circ$}
     \put(0.4,-0.2){$\circ$}
     \put(0.8,-0.2){$\circ$}
     \end{picture}}
\newsavebox{\boxdreipartrotwww}
   \savebox{\boxdreipartrotwww}
   { \begin{picture}(1,1.2)
     \put(0.2,0.65){\line(0,1){0.2}}
     \put(0.2,0.65){\line(1,0){0.5}}
     \put(0.45,0.15){\line(0,1){0.5}}
     \put(0.7,0.65){\line(0,1){0.2}}
     \put(0,0.8){$\circ$}
     \put(0.5,0.8){$\circ$}
     \put(0.25,-0.2){$\circ$}
     \end{picture}}     
\newsavebox{\boxdowndreipartrotwww}
   \savebox{\boxdowndreipartrotwww}
   { \begin{picture}(1,1.2)
     \put(0.2,0.15){\line(0,1){0.2}}
     \put(0.2,0.35){\line(1,0){0.5}}
     \put(0.45,0.35){\line(0,1){0.5}}
     \put(0.7,0.15){\line(0,1){0.2}}
     \put(0,-0.2){$\circ$}
     \put(0.5,-0.2){$\circ$}
     \put(0.25,0.8){$\circ$}
     \end{picture}}          

\newsavebox{\boxsechspartwbwbwb}
   \savebox{\boxsechspartwbwbwb}
   { \begin{picture}(2.5,1)
     \put(0.2,0.2){\line(0,1){0.4}}
     \put(0.6,0.2){\line(0,1){0.4}}
     \put(1,0.2){\line(0,1){0.4}}
     \put(1.4,0.2){\line(0,1){0.4}}
     \put(1.8,0.2){\line(0,1){0.4}}
     \put(2.2,0.2){\line(0,1){0.4}}
     \put(0.2,0.6){\line(1,0){2}}
     \put(0,-0.2){$\circ$}
     \put(0.4,-0.2){$\bullet$}
     \put(0.8,-0.2){$\circ$}
     \put(1.2,-0.2){$\bullet$}
     \put(1.6,-0.2){$\circ$}
     \put(2.0,-0.2){$\bullet$}
     \end{picture}}
 
\newsavebox{\boxcrosspartwbbw}
   \savebox{\boxcrosspartwbbw}
   { \begin{picture}(1,1.4)
     \put(0.25,0.15){\line(1,2){0.4}}
     \put(0.65,0.15){\line(-1,2){0.4}}
     \put(0,0.9){$\circ$}
     \put(0.5,0.9){$\bullet$}
     \put(0,-0.2){$\bullet$}
     \put(0.5,-0.2){$\circ$}
     \end{picture}}
\newsavebox{\boxcrosspartbwwb}
   \savebox{\boxcrosspartbwwb}
   { \begin{picture}(1,1.4)
     \put(0.25,0.15){\line(1,2){0.4}}
     \put(0.65,0.15){\line(-1,2){0.4}}
     \put(0,0.9){$\bullet$}
     \put(0.5,0.9){$\circ$}
     \put(0,-0.2){$\circ$}
     \put(0.5,-0.2){$\bullet$}
     \end{picture}}
\newsavebox{\boxcrosspartwwww}
   \savebox{\boxcrosspartwwww}
   { \begin{picture}(1,1.4)
     \put(0.25,0.15){\line(1,2){0.4}}
     \put(0.65,0.15){\line(-1,2){0.4}}
     \put(0,0.9){$\circ$}
     \put(0.5,0.9){$\circ$}
     \put(0,-0.2){$\circ$}
     \put(0.5,-0.2){$\circ$}
     \end{picture}}
\newsavebox{\boxcrosspartbbbb}
   \savebox{\boxcrosspartbbbb}
   { \begin{picture}(1,1.4)
     \put(0.25,0.15){\line(1,2){0.4}}
     \put(0.65,0.15){\line(-1,2){0.4}}
     \put(0,0.9){$\bullet$}
     \put(0.5,0.9){$\bullet$}
     \put(0,-0.2){$\bullet$}
     \put(0.5,-0.2){$\bullet$}
     \end{picture}}

\newsavebox{\boxhalflibpartwwwwww}
   \savebox{\boxhalflibpartwwwwww}
   { \begin{picture}(1.5,1.4)
     \put(0.3,0.15){\line(1,1){0.8}}
     \put(1.1,0.15){\line(-1,1){0.8}}
     \put(0.7,0.15){\line(0,1){0.8}}     
     \put(0,0.9){$\circ$}
     \put(0.5,0.9){$\circ$}
     \put(1,0.9){$\circ$}
     \put(0,-0.2){$\circ$}
     \put(0.5,-0.2){$\circ$}
     \put(1,-0.2){$\circ$}     
     \end{picture}}

\newsavebox{\boxpositionerd} 
   \savebox{\boxpositionerd}
   { \begin{picture}(2.6,1.5)
     \put(0.9,0.2){\line(0,1){0.9}}
     \put(2.3,0.2){\line(0,1){0.9}}
     \put(0.9,1.1){\line(1,0){1.4}}
     \put(0.05,-0.05){$^\uparrow$}
     \put(0.2,0.4){\tiny$^{\otimes d}$}
     \put(1.35,-0.05){$^\uparrow$}
     \put(1.5,0.4){\tiny$^{\otimes d}$}
     \put(0,-0.2){$\circ$}
     \put(0.7,-0.2){$\circ$}
     \put(1.3,-0.2){$\bullet$}
     \put(2.1,-0.2){$\bullet$}
     \end{picture}}
\newsavebox{\boxpositionerl} 
   \savebox{\boxpositionerl}
   { \begin{picture}(2.6,1.5)
     \put(0.9,0.2){\line(0,1){0.9}}
     \put(2.3,0.2){\line(0,1){0.9}}
     \put(0.9,1.1){\line(1,0){1.4}}
     \put(0.05,-0.05){$^\uparrow$}
     \put(0.2,0.4){\tiny$^{\otimes l}$}
     \put(1.35,-0.05){$^\uparrow$}
     \put(1.5,0.4){\tiny$^{\otimes l}$}
     \put(0,-0.2){$\circ$}
     \put(0.7,-0.2){$\circ$}
     \put(1.3,-0.2){$\bullet$}
     \put(2.1,-0.2){$\bullet$}
     \end{picture}}
\newsavebox{\boxpositionerrpluseins} 
   \savebox{\boxpositionerrpluseins}
   { \begin{picture}(3.9,1.5)
     \put(1.6,0.2){\line(0,1){0.9}}
     \put(3.6,0.2){\line(0,1){0.9}}
     \put(1.6,1.1){\line(1,0){2}}
     \put(0.05,-0.05){$^\uparrow$}
     \put(0.2,0.4){\tiny$^{\otimes r+1}$}
     \put(2.05,-0.05){$^\uparrow$}
     \put(2.2,0.4){\tiny$^{\otimes r-1}$}
     \put(0,-0.2){$\circ$}
     \put(1.4,-0.2){$\bullet$}
     \put(2,-0.2){$\bullet$}
     \put(3.4,-0.2){$\bullet$}
     \end{picture}}
\newsavebox{\boxpositionerdt} 
   \savebox{\boxpositionerdt}
   { \begin{picture}(2.9,1.5)
     \put(1.2,0.2){\line(0,1){0.9}}
     \put(2.9,0.2){\line(0,1){0.9}}
     \put(1.2,1.1){\line(1,0){1.7}}
     \put(0.05,-0.05){$^\uparrow$}
     \put(0.2,0.4){\tiny$^{\otimes dt}$}
     \put(1.65,-0.05){$^\uparrow$}
     \put(1.8,0.4){\tiny$^{\otimes dt}$}
     \put(0,-0.2){$\circ$}
     \put(1,-0.2){$\circ$}
     \put(1.6,-0.2){$\bullet$}
     \put(2.7,-0.2){$\bullet$}
     \end{picture}}
\newsavebox{\boxpositioners} 
   \savebox{\boxpositioners}
   { \begin{picture}(2.6,1.5)
     \put(0.9,0.2){\line(0,1){0.9}}
     \put(2.3,0.2){\line(0,1){0.9}}
     \put(0.9,1.1){\line(1,0){1.4}}
     \put(0.05,-0.05){$^\uparrow$}
     \put(0.2,0.4){\tiny$^{\otimes s}$}
     \put(1.35,-0.05){$^\uparrow$}
     \put(1.5,0.4){\tiny$^{\otimes s}$}
     \put(0,-0.2){$\circ$}
     \put(0.7,-0.2){$\circ$}
     \put(1.3,-0.2){$\bullet$}
     \put(2.1,-0.2){$\bullet$}
     \end{picture}}
\newsavebox{\boxpositionerdinv} 
   \savebox{\boxpositionerdinv}
   { \begin{picture}(2.6,1.5)
     \put(0.9,0.2){\line(0,1){0.9}}
     \put(2.3,0.2){\line(0,1){0.9}}
     \put(0.9,1.1){\line(1,0){1.4}}
     \put(0.05,-0.05){$^\uparrow$}
     \put(0.2,0.4){\tiny$^{\otimes d}$}
     \put(1.35,-0.05){$^\uparrow$}
     \put(1.5,0.4){\tiny$^{\otimes d}$}
     \put(0,-0.2){$\circ$}
     \put(0.7,-0.2){$\bullet$}
     \put(1.3,-0.2){$\bullet$}
     \put(2.1,-0.2){$\circ$}
     \end{picture}}
\newsavebox{\boxpositionersinv} 
   \savebox{\boxpositionersinv}
   { \begin{picture}(2.6,1.5)
     \put(0.9,0.2){\line(0,1){0.9}}
     \put(2.3,0.2){\line(0,1){0.9}}
     \put(0.9,1.1){\line(1,0){1.4}}
     \put(0.05,-0.05){$^\uparrow$}
     \put(0.2,0.4){\tiny$^{\otimes s}$}
     \put(1.35,-0.05){$^\uparrow$}
     \put(1.5,0.4){\tiny$^{\otimes s}$}
     \put(0,-0.2){$\circ$}
     \put(0.7,-0.2){$\bullet$}
     \put(1.3,-0.2){$\bullet$}
     \put(2.1,-0.2){$\circ$}
     \end{picture}}
\newsavebox{\boxpositionerdpluszwei}
   \savebox{\boxpositionerdpluszwei}
   { \begin{picture}(3.4,1.5)
     \put(1.6,0.2){\line(0,1){0.9}}
     \put(3.0,0.2){\line(0,1){0.9}}
     \put(1.6,1.1){\line(1,0){1.4}}
     \put(0.05,-0.05){$^\uparrow$}
     \put(0.2,0.4){\tiny$^{\otimes d+2}$}
     \put(2.05,-0.05){$^\uparrow$}
     \put(2.2,0.4){\tiny$^{\otimes d}$}
     \put(0,-0.2){$\circ$}
     \put(1.4,-0.2){$\bullet$}
     \put(2,-0.2){$\bullet$}
     \put(2.8,-0.2){$\bullet$}
     \end{picture}}
\newsavebox{\boxpositionersminuszwei}
   \savebox{\boxpositionersminuszwei}
   { \begin{picture}(3.4,1.5)
     \put(1,0.2){\line(0,1){0.9}}
     \put(3.0,0.2){\line(0,1){0.9}}
     \put(1,1.1){\line(1,0){2}}
     \put(0.05,-0.05){$^\uparrow$}
     \put(0.2,0.4){\tiny$^{\otimes s}$}
     \put(1.45,-0.05){$^\uparrow$}
     \put(1.6,0.4){\tiny$^{\otimes s-2}$}
     \put(0,-0.2){$\circ$}
     \put(0.8,-0.2){$\bullet$}
     \put(1.4,-0.2){$\bullet$}
     \put(2.8,-0.2){$\bullet$}
     \end{picture}}
\newsavebox{\boxpositionerrnull}
   \savebox{\boxpositionerrnull}
   { \begin{picture}(3.8,1.5)
     \put(1.2,0.2){\line(0,1){0.9}}
     \put(3.4,0.2){\line(0,1){0.9}}
     \put(1.2,1.1){\line(1,0){2.2}}
     \put(0.05,-0.05){$^\uparrow$}
     \put(0.2,0.4){\tiny$^{\otimes r_0}$}
     \put(1.65,-0.05){$^\uparrow$}
     \put(1.8,0.4){\tiny$^{\otimes r_0-2}$}
     \put(0,-0.2){$\circ$}
     \put(1,-0.2){$\bullet$}
     \put(1.6,-0.2){$\bullet$}
     \put(3.2,-0.2){$\bullet$}
     \end{picture}}
\newsavebox{\boxpositionerwbwb}
   \savebox{\boxpositionerwbwb}
   { \begin{picture}(1.8,1)
     \put(0.6,0.2){\line(0,1){0.6}}
     \put(1.4,0.2){\line(0,1){0.6}}
     \put(0.6,0.8){\line(1,0){0.8}}
     \put(0.05,-0.05){$^\uparrow$}
     \put(0.85,-0.05){$^\uparrow$}
     \put(0,-0.2){$\circ$}
     \put(0.4,-0.2){$\bullet$}
     \put(0.8,-0.2){$\circ$}
     \put(1.2,-0.2){$\bullet$}
     \end{picture}}
\newsavebox{\boxpositionerwwbb}
   \savebox{\boxpositionerwwbb}
   { \begin{picture}(1.8,1)
     \put(0.6,0.2){\line(0,1){0.6}}
     \put(1.4,0.2){\line(0,1){0.6}}
     \put(0.6,0.8){\line(1,0){0.8}}
     \put(0.05,-0.05){$^\uparrow$}
     \put(0.85,-0.05){$^\uparrow$}
     \put(0,-0.2){$\circ$}
     \put(0.4,-0.2){$\circ$}
     \put(0.8,-0.2){$\bullet$}
     \put(1.2,-0.2){$\bullet$}
     \end{picture}}
\newsavebox{\boxpositionerwwww}
   \savebox{\boxpositionerwwww}
   { \begin{picture}(1.8,1)
     \put(0.6,0.2){\line(0,1){0.6}}
     \put(1.4,0.2){\line(0,1){0.6}}
     \put(0.6,0.8){\line(1,0){0.8}}
     \put(0.05,-0.05){$^\uparrow$}
     \put(0.85,-0.05){$^\uparrow$}
     \put(0,-0.2){$\circ$}
     \put(0.4,-0.2){$\circ$}
     \put(0.8,-0.2){$\circ$}
     \put(1.2,-0.2){$\circ$}
     \end{picture}}
\newsavebox{\boxpositionerrevwbwb}
   \savebox{\boxpositionerrevwbwb}
   { \begin{picture}(1.8,1)
     \put(0.2,0.2){\line(0,1){0.6}}
     \put(1.0,0.2){\line(0,1){0.6}}
     \put(0.2,0.8){\line(1,0){0.8}}
     \put(0.45,-0.05){$^\uparrow$}
     \put(1.25,-0.05){$^\uparrow$}
     \put(0,-0.2){$\circ$}
     \put(0.4,-0.2){$\bullet$}
     \put(0.8,-0.2){$\circ$}
     \put(1.2,-0.2){$\bullet$}
     \end{picture}}
\newsavebox{\boxpositionersalphaw} 
   \savebox{\boxpositionersalphaw}
   { \begin{picture}(2.6,1.5)
     \put(0.9,0.2){\line(0,1){0.9}}
     \put(2.3,0.2){\line(0,1){0.9}}
     \put(0.9,1.1){\line(1,0){1.4}}
     \put(0.05,-0.05){$^\uparrow$}
     \put(1.35,-0.05){$^\uparrow$}
     \put(0.2,0.4){\tiny$^{\otimes s}$}
     \put(1.5,0.4){\tiny$^{\otimes \alpha}$}
     \put(0,-0.2){$\circ$}
     \put(0.7,-0.2){$\circ$}
     \put(1.3,-0.2){$\circ$}
     \put(2.1,-0.2){$\bullet$}
     \end{picture}}
\newsavebox{\boxpositionersalphab} 
   \savebox{\boxpositionersalphab}
   { \begin{picture}(2.6,1.5)
     \put(0.9,0.2){\line(0,1){0.9}}
     \put(2.3,0.2){\line(0,1){0.9}}
     \put(0.9,1.1){\line(1,0){1.4}}
     \put(0.05,-0.05){$^\uparrow$}
     \put(1.35,-0.05){$^\uparrow$}
     \put(0.2,0.4){\tiny$^{\otimes s}$}
     \put(1.5,0.4){\tiny$^{\otimes \alpha}$}
     \put(0,-0.2){$\circ$}
     \put(0.7,-0.2){$\circ$}
     \put(1.3,-0.2){$\bullet$}
     \put(2.1,-0.2){$\bullet$}
     \end{picture}}

\newsavebox{\boxfatcrosswwww}
   \savebox{\boxfatcrosswwww}
   { \begin{picture}(2,1.6)
     \put(0.2,0.15){\line(0,1){0.2}}
     \put(0.2,0.95){\line(0,1){0.2}}
     \put(0.2,0.95){\line(1,0){0.5}}
     \put(0.2,0.35){\line(1,0){0.5}}
     \put(0.7,0.15){\line(0,1){0.2}}
     \put(0.7,0.95){\line(0,1){0.2}}
     \put(0,1.1){$\circ$}
     \put(0.5,1.1){$\circ$}
     \put(0,-0.2){$\circ$}
     \put(0.5,-0.2){$\circ$}
      \put(0.5,0.35){\line(2,1){1}}
      \put(0.5,0.85){\line(2,-1){1}}
     \put(1.2,0.15){\line(0,1){0.2}}
     \put(1.2,0.95){\line(0,1){0.2}}
     \put(1.2,0.95){\line(1,0){0.5}}
     \put(1.2,0.35){\line(1,0){0.5}}
     \put(1.7,0.15){\line(0,1){0.2}}
     \put(1.7,0.95){\line(0,1){0.2}}
     \put(1,1.1){$\circ$}
     \put(1.5,1.1){$\circ$}
     \put(1,-0.2){$\circ$}
     \put(1.5,-0.2){$\circ$}
     \end{picture}}
\newsavebox{\boxpairpositionerwwwwww}
   \savebox{\boxpairpositionerwwwwww}
   { \begin{picture}(2,1.6)
     \put(0.2,0.95){\line(0,1){0.2}}
     \put(0.2,0.95){\line(1,0){0.5}}
     \put(0.7,0.95){\line(0,1){0.2}}
     \put(0,1.1){$\circ$}
     \put(0.5,1.1){$\circ$}
     \put(0.25,-0.2){$\circ$}
       \put(0.45,0.1){\line(1,1){1}}
      \put(0.5,0.85){\line(2,-1){1}}
     \put(1.2,0.15){\line(0,1){0.2}}
     \put(1.2,0.35){\line(1,0){0.5}}
     \put(1.7,0.15){\line(0,1){0.2}}
     \put(1.25,1.1){$\circ$}
     \put(1,-0.2){$\circ$}
     \put(1.5,-0.2){$\circ$}
     \end{picture}}     

\newsavebox{\boxAspace}
   \savebox{\boxAspace}
   { \begin{picture}(0.1,1.5)
     \end{picture}}
\newsavebox{\boxBspace}
   \savebox{\boxBspace}
   { \begin{picture}(0,1.85)
     \end{picture}}

\newcommand{\idpartww}{\usebox{\boxidpartww}}

\newcommand{\idpartbb}{\usebox{\boxidpartbb}}
\newcommand{\idpartsingletonww}{\usebox{\boxidpartsingletonww}}

\newcommand{\paarpartww}{\usebox{\boxpaarpartww}}
\newcommand{\paarpartbw}{\usebox{\boxpaarpartbw}}
\newcommand{\paarpartwb}{\usebox{\boxpaarpartwb}}
\newcommand{\paarpartbb}{\usebox{\boxpaarpartbb}}

\newcommand{\baarpartbw}{\usebox{\boxbaarpartbw}}
\newcommand{\baarpartwb}{\usebox{\boxbaarpartwb}}

\newcommand{\paarbaarpartwwww}{\usebox{\boxpaarbaarpartwwww}}
\newcommand{\singletonw}{\usebox{\boxsingletonw}}
\newcommand{\singletonb}{\usebox{\boxsingletonb}}

\newcommand{\vierpartwbwb}{\usebox{\boxvierpartwbwb}}

\newcommand{\vierpartwwbb}{\usebox{\boxvierpartwwbb}}

\newcommand{\vierpartrotwbwb}{\usebox{\boxvierpartrotwbwb}}

\newcommand{\crosspartwwww}{\usebox{\boxcrosspartwwww}}

\newcommand{\positionerl}{\usebox{\boxpositionerl}}
\newcommand{\positionerrpluseins}{\usebox{\boxpositionerrpluseins}}

\newcommand{\positionerwbwb}{\usebox{\boxpositionerwbwb}}
\newcommand{\positionerwwbb}{\usebox{\boxpositionerwwbb}}

\newcommand{\crosspart}{\crosspartwwww}

\theoremstyle{plain} 
\newtheorem{thm}{Theorem}[section]
\newtheorem*{thm*}{Theorem}
\newtheorem*{thm*conj*}{Theorem/Conjecture}

\newtheorem{lem}[thm]{Lemma}
\newtheorem*{lem*}{Lemma}
\newtheorem{cor}[thm]{Corollary}
\theoremstyle{definition}
\newtheorem{defn}[thm]{Definition}
\newtheorem*{defn*}{Definition}
\newtheorem{rem}[thm]{Remark}
\newtheorem{ex}[thm]{Example}
\newtheorem{quest}[thm]{Question}
\newtheorem{notation}[thm]{Notation}
\newtheorem*{notation*}{Notation}

\numberwithin{equation}{section}

\renewcommand{\theta}{\vartheta}
\renewcommand{\phi}{\varphi}
\renewcommand{\epsilon}{\varepsilon}
\renewcommand{\subset}{\subseteq}

\newcommand{\N}{\mathbb N}

\newcommand{\C}{\mathbb C}

\newcommand{\norm}[1]{\lVert {#1}\rVert}

\begin{document}
\title{Partition Quantum Spaces}
\author{Stefan Jung and Moritz Weber }
\address{Saarland University, Fachbereich Mathematik, Postfach 151150,
66041 Saarbr\"ucken, Germany}
\email{jung@math.uni-sb.de, weber@math.uni-sb.de}
\date{\today}
\thanks{Both authors were funded by the ERC Advanced Grant NCDFP, held by Roland Speicher. The second author was also funded by the SFB-TRR 195 and the DFG project \emph{Quantenautomorphismen von Graphen}.}

\begin{abstract}
 We propose a definition of partition quantum spaces. They are given by universal \(C^*\)-algebras whose relations come from partitions of sets.
We ask for the maximal compact matrix quantum group acting on them.
We show how those fit into the setting of easy quantum groups:
Our approach yields spaces these groups are acting on.
In a way, our partition quantum spaces arise as the first \(d\) columns of easy quantum groups.
However, we define them as universal \(C^*\)-algebras rather than as \(C^*\)-subalgebras of easy quantum groups.
We also investigate the minimal number \(d\) needed to recover an easy quantum group as the quantum symmetry group of a partition quantum space.
In the free unitary case, \(d\) takes the values one or two.
\end{abstract}

\maketitle

\section{Introduction}

In mathematics, we often have a space \(X\) and want to investigate its symmetries. This leads to the notion of groups. In modern mathematics however, the notion of quantum spaces appeared, for example modelled as possibly non-commutative \(C^*\)-algebras. Asking for the quantum symmetries of such topological quantum spaces leads to the definition of quantum groups.
Our work is based on the definition of (\(C^*\)-algebraic) compact matrix quantum groups by S.L. Woronowicz in \cite{woronowiczpseudogroups}. Roughly speaking, such a quantum group consists of a unital \(C^*\)-algebra \(\mathcal{A}\) generated by the entries of a matrix \(u_G=(u_{ij})\) such that there exists a comultiplication \(\Delta:\mathcal{A}\rightarrow\mathcal{A}\otimes \mathcal{A}\) (see Definition \ref{defn:CMQG}). Now there are two fundamental questions regarding quantum symmetries:
\begin{compactitem}
\item Given a quantum space \(X\) -- what is its quantum symmetry group?
\item Conversely, given a quantum group \(G\) -- can we find a quantum space, such that its quantum symmetry group is precisely \(G\)?
\end{compactitem}
In the present article, we mainly deal with the second kind of questions.

In \cite{banicaspeicherliberation}, T. Banica and R. Speicher introduced an important class of compact matrix quantum groups, so called easy quantum groups. Their structure is encoded by partitions of sets via Woronowicz's Tannaka-Krein duality from \cite{woronowicztannakakrein}.
More precisely, one can associate linear maps with given partitions and ask them to span the intertwiner spaces of some compact matrix quantum group.
These special quantum groups were generalized in \cite{tarragoweberclassificationpartitions,tarragoweberclassificationunitaryQGs} by P. Tarrago and the second author:
If \(N\!\in\!\N\) and \(\Pi\) is a (suitable) set of two-coloured partitions, we can associate with every partition \(p\!\in\!\Pi\) relations \(R^{Gr}_p(u_G)\) for some generators \(\big(u_{ij}\big)_{1\!\le i,j\le N}\) and define the \emph{easy quantum group} \(G_N(\Pi)\) via the universal \(C^*\)-algebra
\[C(G_N(\Pi)):=C^*\big(u_{ij}\;|\;\forall p\!\in\!\Pi:\textnormal{ the relations } R^{Gr}_p(u_G)\textnormal{ hold}\big).\]
Using this machinery, many compact matrix quantum groups may be produced.
 
The question is now on which quantum spaces these quantum groups act, or even stronger: Find a quantum space, such that a given eas quantum group \(G_N(\Pi)\) describes its quantum symmetries, i.e. \(G_N(\Pi)\) is its \emph{quantum symmetry group} (in the sense of Definition \ref{defn:QSymG}).

In some cases answers were already found: The quantum permutation group \(S_N^ +\), for example, is the quantum symmetry group of \(N\) quantum points (see \cite{wang1998}) and the orthogonal quantum group \(O_N^+\) is the quantum symmetry group of the free real sphere (see \cite{banicagoswami_noncommutative_spheres,wang1995}).
Moreover, every compact matrix quantum group acts on the first column of its fundamental representation \(u_G\), i.e. on the \(C^*\)-algebra \(C^*(u_{11},u_{21},\ldots,u_{N1})\).
But in contrast to this, given only the first row, the quantum symmetry group of this space is not always described by the quantum group we started with (see Example \ref{ex:S'_N_2}). In this sense we cannot recover in general a given easy quantum group from its first column. 

For easy quantum groups \(G_N(\Pi)\) we develop this idea of a ``first column space'' further into two directions.
Firstly, instead of using directly the entries \(u_{i1}\) in the first column of \(u_{G_N(\Pi)}\) we define our quantum spaces as universal \(C^*\)-algebras generated by the entries of a vector \(x\!=\!(x_1,\ldots,x_N)^T\).
The relations \(R^{Sp}_p(x)\) for these generators are motivated by the first column of \(u_{G_N(\Pi)}\) (see Definitions \ref{defn:quantum_space_relations_for_the_x_i} and \ref{defn:PQS_vector_case}).
Given \(N\!\in\!\N\), a set of partitions \(\Pi\) and a vector \(x\) as above, we then define
\[C\big(X_{N}(\Pi)\big):=C^*\big(x_1,\ldots,x_N\;|\;\forall p\!\in\!\Pi:\textnormal{ the relations } R^{Sp}_p(x)\textnormal{ hold})\]
and call \(X_{N}(\Pi)\) a \emph{partition quantum space (PQS) of one vector}.

Secondly, we do not consider only one column but rather \(d\) columns of \(u_G\) at once. In this sense we generalize the definition above for a tupel of vectors
\[x:=\left(\begin{pmatrix}x_{11}\\ \vdots\\ x_{N1}\end{pmatrix},\ldots,\begin{pmatrix}x_{1d}\\ \vdots\\ x_{Nd}\end{pmatrix}\right),\]
producing a \emph{partition quantum space (PQS) of \(d\) vectors}, \(X_{N,d}(\Pi)\).

Given a set of partitions \(\Pi\) and \(d\!\le\!N\!\in\!\N\), we have (under some mild conditions) an associated easy quantum group \(G_N(\Pi)\) and a quantum space \(X_{N,d}(\Pi)\). In this work we concentrate on the question, how these two objects fit together in the sense of quantum group actions.
\begin{quest}\label{quest:does_BSQG_act_on_PQS-introduction}
Does the easy quantum group \(G_N(\Pi)\) act on \(X_{N,d}(\Pi)\)?
\end{quest}
\begin{quest}\label{quest:is_BSQG_QSymG_of_PQS_introduction}
Is \(G_N(\Pi)\) the quantum symmetry group of \(X_{N,d}(\Pi)\)?
\end{quest}
\begin{quest}\label{quest:minimal_d}
What is the smallest \(d\) if we want to recover the easy quantum group \(G_N(\Pi)\) as the quantum symmetry group of \(X_{N,d}(\Pi)\)?
\end{quest}

At last we want to mention some other works touching our topic.
P. Podle\'{s}'s definition of \emph{quantum spheres} in \cite{podlesquantumspheres} was a first but important step in quantizing the notion of a classical space.
Authors like T. Banica, J. Bhowmick, D. Goswami, P. Podle\'{s}, A. Skalski and Sh. Wang investigated various quantum spaces and actions of quantum groups on them and asked (for example under the name of \emph{quantum isometry groups}) for the universal objects acting on these spaces (see \cite{podlesquantumspheres,podles1995,wang1998,goswamiquantumgroupofisometries1,bhowmickgoswamiqisog_examples_and_computations,
bhowmickgoswamiskalski,banicagoswami_noncommutative_spheres,bhowmickgoswamiquantumgroupriemannian}).
The idea of a quantum space inspired by one or several rows/columns of a compact matrix quantum group \(G\) can be found for example in \cite{banicaskalskisoltan_homogeneous_spaces}, but note that the spaces there are defined via \(C^*\)-subalgebras of \(C(G)\), whereas we introduce them as universal \(C^*\)-algebras.
At last we mention the recent work \cite{banicageometricaspects} by T. Banica, where partition induced relations similar to those in our article are used to describe certain quantum subspaces of the free complex sphere.
In contrast to the setting presented there, where  it is part of the assumptions that an easy quantum group is the quantum symmetry group of a suitable quantum space, this is the central question in our work.
Additionally, as mentioned above, we generalize the idea of quantum vectors to tupels of quantum vectors.

\section{Main results}
Let \(d,N\!\in\!\N\) with \(d\!\le\!N\) and let \(\Pi\) be a set of partitions defining both an easy quantum group \(G_N(\Pi)\) and a partition quantum space \(X_{N,d}(\Pi)\).
The relations  on the generators \(x_{ij}\) of our quantum spaces can be seen as derived from the first \(d\) columns of the matrix \(u_{G_N(\Pi)}=(u_{ij})\) (see the appendix for an overview on all relations associated to partitions).
\begin{thm*} [see Theorems \ref{thm:hom_from_PQS_into_BSQG_vector_case} and \ref{thm:hom_from_PQS_into_BSQG_d-vector_case}]
We have a \(^*\)-homomorphism 
\[\phi:\;C\big(X_{N,d}(\Pi)\big)\rightarrow C\big(G_N(\Pi)\big);\;x_{ij}\mapsto u_{ij}\;\quad,\quad1\!\le\!i\!\le\!N, 1\!\le\!j\!\le\!d\]
mapping the entries of \(x\) canonically onto the first \(d\) columns of \(u_{G_N(\Pi)}\).
\end{thm*}
Note that we do not know whether \(\phi\) is an isomorphism or not.

Furthermore, we answer Question \ref{quest:does_BSQG_act_on_PQS-introduction} in full generality:
\begin{thm*}[see Theorem \ref{thm:existence_of_alpha_and_beta}]
For any \(1\!\le\!d\!\le\!N\), \(G_N(\Pi)\) acts on \(X_{N,d}(\Pi)\) from the left and right.
\end{thm*}

In the case \(d\!=\!N\) we can also answer Question \ref{quest:is_BSQG_QSymG_of_PQS_introduction}:
\begin{thm*}[see Corollary \ref{cor:X_is_BSQS_if_d=N}]
\(G_N(\Pi)\) is the quantum symmetry group of \(X_{N,N}(\Pi)\).
\end{thm*}
The above result leads to a new question (compare Question 1.3): Given a set \(\Pi\), which is the minimal number \(d\) such that the latter theorem stays true?
We cannot answer this question as generally as the ones before, so we restrict to the case of \emph{non-crossing} partitions, corresponding to \emph{free easy quantum groups}. The possible quantum groups are completely classified in \cite{tarragoweberclassificationpartitions,tarragoweberclassificationunitaryQGs} and we can describe them by suitable sets \(\Pi\) of partitions (see Table \ref{table:classification_free_case} in Section \ref{sec:free_case}). Fixing these sets of partitions, we can bound the necessary \(d\) to at most 2:
\begin{thm*}[see Theorem \ref{thm:free_case_main_result}]
Let \(\Pi\) be a set of non-crossing partitions from Table \ref{table:classification_free_case}.
\begin{itemize}
\item[(i)]
For \(d\!=\!2\), the easy quantum group \(G_N(\Pi)\) is the quantum symmetry group of the partition quantum space \(X_{N,d}(\Pi)\).
\item[(ii)]If \(\Pi\) generates a \emph{blockstable category of partitions}, then (i) even holds for \(d\!=\!1\). In particular we can reconstruct the easy quantum group from the first column of its fundamental unitary.
\end{itemize}
\end{thm*}

We conjecture (see open questions in Section \ref{sec:open_questions}) that in the non-blockstable case the situation \(d\!=\!1\) does not work, i.e. that amongst all categories of non-crossing partitions the question of blockstability is equivalent to the minimal \(d\) being equal to one.
 
\section{Preliminaries}

In the context of \(C^*\)-algebras we denote by \(\otimes\) the minimal tensor product and we define \([n]:=\{1,\ldots,n\}\subseteq\N\) for all \(n\!\in\!\N\).\vspace{11pt}

In this section we give a very brief introduction to partitions of sets and how we can associate easy quantum groups to them (also known as \emph{easy quantum groups}). For more details see \cite{banicaspeicherliberation} and \cite{tarragoweberclassificationpartitions,tarragoweberclassificationunitaryQGs,weber_LNM,weber_PMS}.
Moreover, we introduce a new kind of decomposition of labelings of partitions. 
Finally, we present a quantum version of matrix-vector actions.

\subsection{Two-coloured partitions}\label{subsec:partitions}

A \emph{(two-coloured) partition} on \(k\) upper and \(l\) lower points is a partition of the ordered set \([k+l]\) into non-empty, disjoint subsets, where each element gets a label 1 (white) or * (black). The subsets of the partition are called \emph{blocks}.
We may illustrate such partitions by lines representing the blocks:
\begin{equation}\label{eqn:partition_example}
\setlength{\unitlength}{0.3cm}
p=
\begin{picture}(5,2.4)
\put(0,-1.8) {$\bullet$}
 \put(1,-1.8) {$\bullet$}
 \put(2,-1.8) {$\circ$}
 \put(0,1.6) {$\circ$}
 \put(1,1.6) {$\circ$}
 \put(2,1.6) {$\bullet$}
\put(-0.11,2.6){\partii{1}{0}{1}}
\put(-0.12,2.6){\parti{2}{2}}
\put(-0.12,-1){\uppartii{1}{1}{2}}
\put(-0.12,-1){\upparti{1}{0}}
\end{picture}\vspace{8pt}
\end{equation}
If a block contains upper and lower points, we call it a \emph{through-block}. The number of through-blocks in \(p\) is denoted by \(tb(p)\).
A partition is called \emph{non-crossing}, if the lines in the corresponding picture do not cross.
The set \(\mathcal{P}(k,l)\) contains all partitions with \(k\) upper and \(l\) lower points (in all possible labelings) and \(\mathcal{P}\!:=\!\bigcup_{k,l\in\N_0}\mathcal{P}(k,l)\) is the union of all those sets. For given words \(\omega\!\in\!\{1,*\}^k\) and \(\omega'\!\in\!\{1,*\}^l\) we denote by \(\mathcal{P}(\omega,\omega')\!\subseteq\!\mathcal{P}(k,l)\) all partitions with upper point labeling according to \(\omega\) and lower point labeling according to \(\omega'\) (from left to right, respectively).

We have some operations defined on \(\mathcal{P}\). Given \(p\!\in\!\mathcal{P}(\omega^{(1)},\omega^{(2)})\) and \(q\!\in\!\mathcal{P}(\omega^{(3)},\omega^{(4)})\) there exists a tensor product \(p\otimes q\), an involution \(p^*\) and (if \(\omega^{(2)}\!=\!\omega^{(3)}\)) a composition \(qp\). A set of partitions closed under these operations and containing \(\{\raisebox{-2pt}{\scalebox{0.7}{\idpartww,\idpartbb}\scalebox{0.9}{,\paarpartwb,\paarpartbw,\baarpartwb,\baarpartbw}}\}\) is called a \emph{category of partitions}. Here is an example of these operations using \(p\) as in Equation \ref{eqn:partition_example}:
\begin{equation*}
\setlength{\unitlength}{0.3cm}
pp^*=
\raisebox{0.6cm}{\rotatebox{180}{\reflectbox{
\begin{picture}(3,2)
\put(0,-1.8) {$\bullet$}
 \put(1,-1.8) {$\bullet$}
 \put(2,-1.8) {$\circ$}
\put(-0.11,2.6){\partii{1}{0}{1}}
\put(-0.12,2.6){\parti{2}{2}}
\put(-0.12,-1){\uppartii{1}{1}{2}}
\put(-0.12,-1){\upparti{1}{0}}
\end{picture}
}}}
\hspace{-1.18cm}
\raisebox{-0.54cm}{
\begin{picture}(3,2)
\put(0,-1.8) {$\bullet$}
 \put(1,-1.8) {$\bullet$}
 \put(2,-1.8) {$\circ$}
 \put(0,1.6) {$\circ$}
 \put(1,1.6) {$\circ$}
 \put(2,1.6) {$\bullet$}
\put(-0.11,2.6){\partii{1}{0}{1}}
\put(-0.12,2.6){\parti{2}{2}}
\put(-0.12,-1){\uppartii{1}{1}{2}}
\put(-0.12,-1){\upparti{1}{0}}
\end{picture}
}
=
\begin{picture}(3,4)
 \put(0,-1.5) {$\bullet$}
 \put(1,-1.5) {$\bullet$}
 \put(2,-1.5) {$\circ$}
 \put(0,1.4)  {$\bullet$}
 \put(1,1.4)  {$\bullet$}
 \put(2,1.4)  {$\circ$}
 
\put(1.35,-0.7){\line(0,1){0.7}}
\put(1.35,0.0){\line(1,0){1}}
\put(2.35,-0.7){\line(0,1){0.7}}

\put(1.35,0.6){\line(0,1){0.7}}
\put(1.35,0.6){\line(1,0){1}}
\put(2.35,0.6){\line(0,1){0.7}}

\put(1.83,0.0){\line(0,1){0.6}}

\put(0.34,-0.7){\line(0,1){0.7}}
\put(0.34,0.6){\line(0,1){0.7}}
\end{picture}
\vspace{11pt}\vspace{11pt}
\end{equation*}
See \cite{tarragoweberclassificationpartitions} or \cite[Appendix B]{weber_partition_cstar-algebras_I} for more on two-coloured partitions and more examples.

\subsection{Labeling of partitions}\label{subsec:labeling_of_partitions}
\begin{notation}
Let \(p\!\in\!\mathcal{P}(k,l)\). Every pair of multi indices \((t,t')\!\in\!\N^k\!\times\!\N^l\) gives rise to a labeling of the points of \(p\) by labeling the upper points from left to right by \(t=(t_1,\ldots,t_k)\) and likewise the lower points by \(t'=(t'_1,\ldots,t'_l)\). A labeling is \(\emph{valid}\) if for every block all of its points have the same label. We can also  speak of valid labelings of a subset of points if in this subset connected points are labeled equally.
\end{notation}
\begin{defn}\label{defn:delta_p}
Every partition \(p\!\in\!\mathcal{P}(k,l)\) gives rise to a function
\[\delta_p: \N^k\!\times\!\N^l\rightarrow\{0,1\}\;;\;\delta_p(i,j)=
\begin{cases}
1 & (i,j)\textnormal{ is a valid labeling of }p,\\
0 & \textnormal{otherwise.}
\end{cases}
\]
\end{defn}
As an example, consider the partition \(p\) from Equation \ref{eqn:partition_example}. A labeling \(t\!=\!(t_1,t_2,t_3)\) of the upper row is valid if \(t_1\!=\!t_2\). Likewise \(t'=(t'_1,t'_2,t'_3)\) is valid for the lower row if \(t'_2\!=\!t'_3\). The pair \((t,t')\) is valid for \(p\)  if additionally \(t_3\!=\!t_2'\) holds, i.e. the labelings on the through-block fit together. 

\subsection{Decomposition of labelings}
We use the observation above to decompose given sets of multi indices into disjoint subsets, due to their validity as labelings.
\begin{notation}\label{not:partitioning_of_labelings}
Given \(p\!\in\!\mathcal{P}(k,l)\) and \(N\!\in\!\N\) we can decompose the sets \([N]^k\) and \([N]^l\) in the following way:
\[[N]^k=T_0\,\dot{\cup}\,T_1\,\dot{\cup}\,\ldots\,\dot{\cup}\,T_r\quad;\quad[N]^l=T'_0\,\dot{\cup}\,T'_1\,\dot{\cup}\,\ldots\,\dot{\cup}\,T'_r,\]
such that 
\begin{compactitem}
\item [(i)] \(r=N^{tb(p)}\), where \(tb(p)\) denotes the number of through-blocks of \(p\),
\item [(ii)]\(T_0\) and \(T'_0\) are the invalid labelings of the upper (respectively lower) row,
\item [(iii)]for every \(1\!\le\!i\!\le\!r\) every labeling \((t,t')\!\in\!T_i\!\times\!T'_i\) is valid,
\item[(iv)] for every \(1\!\le\!i\!\le\!r\) the sets \(T_i\) and \(T'_i\) are non-empty.
\item [(v)]if \((t,t')\!\in\![N]^k\!\times\![N]^l\) is a valid labeling, then \((t,t')\!\in\!T_i\!\times\!T'_i\)  for some \(1\!\le\!i\!\le\!r\),
\item [(vi)]for every \(1\!\le\!i\!\le\!r\) and \((t,t'),(s,s')\!\in\!T_i\!\times\!T'_i\) we have that \((t,t')\) labels the through-block points of \(p\) the same way as \((s,s')\) does.
\end{compactitem}
The listed properties above are partially redundant. For example (iv)--(vi) follow from (i)--(iii). 
\end{notation}
\begin{rem}\label{rem:case_k,l=0}
Note the special case of \(0\!\in\!\{k,l\}\): An empty row has only one possible labeling (which is valid), namely the empty word \(\epsilon\!\in\![N]^0\). So if for example a partition has only lower points, then \(r=1\), \(T_1=\{\epsilon\}\) and \(T_0\) is empty.

Furthermore note, that \(|T_i|=|T_j|\) and \(|T'_i|=|T'_j|\) for all \(1\!\le\!i,j\!\le\!r\) as the possibilities to extend a valid through-block labeling to a valid labeling of the whole row does not depend on the actual through-block labeling.
\end{rem}

\begin{lem}
Decompositions of \([N]^k\) and \([N]^l\) as in Notation \ref{not:partitioning_of_labelings} exist and they are unique up to permutations of the index set \(\{i\;|\;1\!\le\!i\!\le\!r\}\).
\end{lem}
\begin{proof}
\emph{Existence:} Define \(T_0\) and \(T'_0\) as described in (ii). As \(p\) has \(tb(p)\) through-blocks we have \(r=[N]^{tb(p)}\) possibilities to label the through-block points in a valid way.
Numbering these possibilities from 1 to \(r\) we can take any \(1\!\le\!i\!\le\!r\) and extend it to labelings of the whole partition. This defines the sets \(T_i\) and \(T'_i\):
\begin{align*}
T_i:=&\{\textnormal{all valid labelings of the upper row of }p\textnormal{ with through-block labeling }i\}\\
T'_i:=&\{\textnormal{all valid labelings of the lower row of }p\textnormal{ with through-block labeling }i\}
\end{align*}
It is now easy to check, that the properties (i)--(vi) are fulfilled:
(i) and (ii) hold by construction.
(iii) is true as given labelings \(t\!\in\!T_i\) and \(t'\!\in\!T'_i\) are valid for their respective row and the through-block labelings fit together as both \(t\) and \(t'\) arise from a common through-block labeling \(i\).
Obviously (iv) is true, as we can always extend a valid through-block labeling by labeling all remaining points the same.
For property (v) note, that a valid labeling \((t,t')\) always restricts to a valid labeling of the through-block points.
If this through-block labeling corresponds to \(i\!\in\![r]\), then \(t\) and \(t'\) appear in the construction of \(T_i\) and \(T'_i\), respectively.
Property (vi) is fulfilled, as by construction \((t,t')\) and \((s,s')\) arise from the same through-block labeling.
Finally we check \([N]^k=T_0\dot{\cup}\ldots\dot{\cup}T_r\) (the proof of \([N]^l=T'_0\,\dot{\cup}\,T'_1\,\dot{\cup}\,\ldots\,\dot{\cup}\,T'_r\) is analogous):
By construction only \(T_0\) contains non-valid labelings of the upper row and it contains all of them.
Every valid upper row labeling appears as it restricts to a valid labeling of the upper through-blocks, which can be extended to a valid through-block labeling \(i\!\in\![r]\) of both rows.
On the other side the \(T_1,\ldots,T_r\) are disjoint as different through-block labelings \(1\!\le\!i\!\neq\!j\!\le\!r\) always differ when restricted to only one row, so \(T_i\cap T_j=\emptyset\).
\newline
\emph{Uniqueness:}
Of course \(T_0\) and \(T'_0\) are uniquely defined.
Consider now two valid labelings \(t\) and \(s\) of the upper row.
Assume that they do not restrict to the same labeling of upper through-block points but are contained in the same \(T_i\). Then for any \(t'\!\in\!T'_i\) -- we have \(T'_i\!\neq\!\emptyset\) by (iv) -- the labelings \((t,t')\) and \((s,t')\) were valid for the whole partition by (iii).
This is a contradiction, as \(t'\) uniquely determines the upper through-block labelings both of \(t\) and \(s\).
As there are \(N^{tb(p)}\) pairwise different valid labelings of the upper through-block points and \(r\!=\!N^{tb(p)}\) by property (i), the sets \(T_1,\ldots,T_r\) must be the (pairwise different) equivalence classes of valid upper row labelings with respect to the relation ``equality on through-block points''.
Having now the sets \(T_1\ldots,T_r\) (and likewise \(T'_1,\ldots,T'_r\)) at hand, property (iii) says that \(T'_i\) must correspond to the same through-block labeling as \(T_i\). So up to (simultaneous) permutations of the index set \(\{i\;|\;1\!\le\!i\!\le\!r\}\) we have uniqueness as claimed.
\end{proof}

\begin{ex}\label{ex:decomposition_of_labelings}
Consider again the partition \(p\) from Equation \ref{eqn:partition_example}, so \(r\!=\!N^{tb(p)}\!=\!N\). A pair \((T_i,T'_i)\) for \(1\!\le\!i\!\le\!N\) corresponds to a distinct valid labeling of the through-block points (i.e. \(t_3\!=\!t'_2\!=\!t'_3\)). So we have
\begin{align*}
T_0=\{(t_1,t_2,t_3)\!\in\![N]^3\;|\; t_1\!\neq\!\!t_2\}\quad&,\quad T'_0=\{(t'_1,t'_2,t'_3)\!\in\![N]^3\;|\;t'_2\!\neq\!\!t'_3\}\\
T_i=\{(t,t,i)\!\in\![N]^3\;|\;t\!\in\![N]\}\quad\;\;\;&,\quad T'_i=\{(t',i,i)\!\in\![N]^3\;|\;t'\!\in\![N]\}\quad \textnormal{ for }1\!\le\!i\!\le\!N.
\end{align*}
\end{ex}

\subsection{Compact matrix quantum groups}

In this work only one class of quantum groups is relevant, namely \emph{compact matrix quantum groups (CMQGs)} as defined by S.L. Woronowicz in \cite{woronowiczpseudogroups}:
\begin{defn}\label{defn:CMQG}
Let \(N\!\in\!\N\) and \(u_G:=(u_{ij})\) be an \(N\!\times\!N\)-matrix with entries in some unital \(C^*\)-algebra \(\mathcal{A}\). Assume the following:
\begin{itemize}
\item[(i)] The entries \(u_{ij}\) generate \(\mathcal{A}\) as a \(C^*\)-algebra.
\item[(ii)] The matrices \(u_G\) and \(\bar{u}_G:=(u^*_{ij})\) are invertible.
\item[(iii)] There is a  \(^*\)-homomorphism \(\Delta:\mathcal{A}\rightarrow\mathcal{A}\otimes\mathcal{A}\) called \emph{co-multiplication}, fulfilling  
\[\Delta(u_{ij})=\sum_{k=1}^{N}u_{ik}\otimes u_{kj}.\]
\end{itemize}
Then we denote \(\mathcal{A}\) also by \(C(G)\) and call it the \emph{non-commutative functions} on an (abstractly given) \emph{compact matrix quantum group (CMQG) \(G\)}.
\end{defn}
The reason for this name is, that in the case of a commutative \(C^*\)-algebra \(\mathcal{A}\), the object \(G\) really is a matrix group and the \(u_{ij}\) are the coordinate functions in the algebra \(C(G)\) of continuous functions on \(G\). See \cite{neshveyevtuset,timmermann,weber_PMS} for more on CMQGs.

\subsection{Easy quantum groups}

\begin{defn}\label{defn:quantum_group_relations_using_T_i's}
Let \(N\!\in\!\N\), \(u\!:=\!(u_{ij})\) an \(N\!\times\!N\)-matrix of generators and \(p\!\in\!\mathcal{P}(\omega,\omega')\!\subseteq\!\mathcal{P}(k,l)\) be a partition. Using the Notations \ref{not:partitioning_of_labelings}, we associate the following relations to the \(u_{ij}\), denoted by \(R^{Gr}_p(u)\):
\begin{itemize}
\item[(i)]
\(\displaystyle\sum_{t\in T_i} u_{t_1\gamma_1}^{\omega_1}\cdots u_{t_k\gamma_k}^{\omega_k}
=\sum_{t'\in T'_j} u_{\gamma'_1t'_1}^{\omega'_1}\cdots u_{\gamma'_lt'_l}^{\omega'_l}\quad,\quad 1\!\le \!i,j\!\le\!r\), \(\gamma\!\in\!T_{j}\textnormal{ and }\gamma'\!\in\!T'_{i}.\)

\item[(ii)]
\(\displaystyle\sum_{t\in T_i}u_{t_1\gamma_1}^{\omega_1}\cdots u_{t_k\gamma_k}^{\omega_k}=0\quad,\quad 1\!\le\!i\!\le\!r\textnormal{ and }\gamma\!\in\!T_0.\)

\item[(iii)]
\(\displaystyle\sum_{t'\in T'_j}u_{\gamma'_1t'_1}^{\omega'_1}\cdots u_{\gamma'_lt'_l}^{\omega'_l}=0\quad,\quad 1\!\le\!j\!\le\!r\textnormal{ and }\gamma'\!\in\!T'_0.\)
\end{itemize}
\end{defn}

Note for (i) (compare Remark \ref{rem:case_k,l=0}), that in case of \(0\!\in\!\{k,l\}\) the sum on the corresponding side is equal to \(\mathds{1}\), corresponding to the empty word \(\epsilon\).

\begin{defn}\label{defn:BSQG}
Let \(N\!\in\!\N\) be given as well as a set \(\Pi\) of partitions including the four \emph{mixed-coloured pair partitions} \(\{\paarpartwb,\paarpartbw,\baarpartwb,\baarpartbw\}\). Let further \(u_{G_N(\Pi)}:=(u_{ij})\) be an \(N\!\times\!N\)-matrix of generators. Then we define the universal \(C^*\)-algebra 
\[C\big(G_N(\Pi)\big):=C^*(u_{ij}\;|\;\forall p\!\in\!\Pi:\textnormal{ the relations }R^{Gr}_p(u_{G_N(\Pi)})\textnormal{ hold})\]
and call it the \emph{non-commutative functions} on the \emph{easy quantum group \(G_N(\Pi)\)}. 
\end{defn}

\begin{rem}
Note that the relations \(\big\{R_p^{Gr}(u_{G})\;|\;p\!\in\!\{\paarpartwb,\paarpartbw,\baarpartwb,\baarpartbw\}\big\}\) are equivalent to the fact that \(u_Gu_G^*\), \(\bar{u}_G\bar{u}_G^*\), \(\bar{u}_G^*\bar{u}_G\) and \(u_G^ *u_G\), respectively, are equal to \(\mathds{1}\), i.e. \(u_G\) and \(\bar{u}_G\) are unitaries. So the theory of easy quantum groups is a theory of unitary quantum matrices.
\end{rem}

\begin{rem}
From the perspective of \cite{banicaspeicherliberation} and \cite{tarragoweberclassificationunitaryQGs} the definitions above are a reformulation of the original ones, adapted to our purposes. We outline some key ideas of the theory presented there.
\begin{itemize}
\item [(i)]To every set of partitions \(\Pi\) as above we can consider the category of partitions \(\mathcal{C}:=\langle\Pi\rangle\)  generated by \(\Pi\).
\item [(ii)]Every partition \(p\!\in\!\mathcal{P}(k,l)\cap\mathcal{C}\) corresponds to a linear map \(T_p:\left(\C^N\right)^{\otimes k}\!\rightarrow\! \left(\C^N\right)^{\otimes l}\). Here the delta-function \(\delta_p\) from Definition \ref{defn:delta_p} plays a central role.
\item [(iii)]The linear span of the maps \(\big(T_p\big)_{p\in\mathcal{C}}\) is a \emph{concrete monoidal \(W^*\)-category}.
\item [(iv)]By Tannaka-Krein duality for CMQGs, see \cite{woronowicztannakakrein}, every such category produces a CMQG, \(G_N(\mathcal{C})\), such that the intertwiner spaces
\[Mor_{G_N(\Pi)}(\omega,\omega'):=\{T:\left(\C^N\right)^{\otimes |\omega|}\rightarrow \left(\C^N\right)^{\otimes |\omega'|}\textnormal { linear}\;|\;Tu^{\otimes\omega}=u^{\otimes\omega'}T\}\]
coincide with the linear spans of the maps \(\left(T_p\right)_{p\in\mathcal{P}(\omega,\omega')}\).

\item [(v)]The construction of \(G_N(\mathcal{C})\) is straightforward: Every equation \(T_pu^{\otimes\omega}=u^{\otimes\omega'}T_p\) can be seen as relations for the matrix entries \(u_{ij}\). The universal \(C^*\)-algebra generated by the \(u_{ij}\) and all these relations turns out to be the object \(C\left(G_N(\mathcal{C})\right)\).
\item [(vi)]Using the definitions of the maps \(T_p\) from \cite{tarragoweberclassificationunitaryQGs}, the relations \(T_pu^{\otimes\omega}=u^{\otimes\omega'}T_p\), associated to a partition \(p\!\in\!\mathcal{P}(\omega,\omega')\!\subseteq\!\mathcal{P}(k,l)\) read as
\begin{equation}\label{eqn:partition-induced_relations_for_the_u_ij_preliminaries}
\sum_{t\in[N]^k}\delta_p(t,\gamma')u_{t_1\gamma_1}^{\omega_1}\cdots u_{t_k\gamma_k}^{\omega_k}
=\sum_{t'\in[N]^l}\delta_p(\gamma,t')u_{\gamma'_1t'_1}^{\omega'_1}\cdots u_{\gamma'_lt'_l}^{\omega'_l}
\end{equation}
for every \((\gamma,\gamma')\!\in\![N]^k\!\times\![N]^l\).
The behaviour of \(\delta_p\) on \([N]^k\!\times\![N]^l\) is encoded by the decomposition into the subsets \(T_i\) and \(T'_i\), see Notation \ref{not:partitioning_of_labelings}.
So Equation \ref{eqn:partition-induced_relations_for_the_u_ij_preliminaries} yields exactly the relations \(R^{Gr}_p(u)\) from Definition \ref{defn:quantum_group_relations_using_T_i's}. Indeed observe that for \(\gamma\!\in\!T_j\), \(\gamma'\!\in\!T'_i\) and \(1\!\le\!i,j\!\le\!r\) we have
\[\sum_{t\in[N]^k}\delta_p(t,\gamma')u_{t_1\gamma_1}^{\omega_1}\cdots u_{t_k\gamma_k}^{\omega_k}
=\sum_{t\in T_i}u_{t_1\gamma_1}^{\omega_1}\cdots u_{t_k\gamma_k}^{\omega_k}\]
and
\[\sum_{t'\in[N]^l}\delta_p(\gamma,t')u_{\gamma'_1t'_1}^{\omega'_1}\cdots u_{\gamma'_lt'_l}^{\omega'_l}
=\sum_{t'\in T'_j}u_{\gamma'_1t'_1}^{\omega'_1}\cdots u_{\gamma'_lt'_l}^{\omega'_l}\]
as \(\delta_p(t,\gamma')=1\Leftrightarrow t\!\in\!T_i\) and \(\delta_p(\gamma,t')=1\Leftrightarrow t'\!\in\!T'_j\).
\end{itemize}
\end{rem}
As we will also use it later on (see Lemma \ref{lem:first_column-relations} and Theorem \ref{thm:free_case_main_result}), we separately mention also the following result.
\begin{lem}\label{lem:relations_on_generators_imply_relations_for_category}
Let \(u_G:=(u_{ij})\) be a matrix of generators associated to a compact matrix quantum group \(G\). Let \(\Pi\) be a set of partitions including  \(\{\paarpartwb,\paarpartbw,\baarpartwb,\baarpartbw\}\) such that the relations \(R^{Gr}_p(u_G)\) hold for all \(p\!\in\!\Pi\). Then also \(R^{Gr}_{q}(u_G)\) hold for every \(q\) in the category \(\mathcal{C}\!=\langle\Pi\rangle\) generated by \(\Pi\). In particular, \(R^{Gr}_q(u_G)\) holds for \(q\!=\!pp^*\), \(q=p^*\) and for any rotated version \(q\!=\!\textnormal{rot}(p)\) of \(p\) (see \cite{tarragoweberclassificationpartitions,weber_partition_cstar-algebras_I}.
\end{lem}
Originally, easy quantum groups were defined as exactly those CMQGs \(G\), whose intertwiner spaces are given via categories of partitions, but due to Lemma \ref{lem:relations_on_generators_imply_relations_for_category} we only need to consider suitable generating sets \(\Pi\) of partitions in order to completely understand the universal \(C^*\)-algebras \(C(G)\). This allows us to write \(G_N(\Pi)\) instead of \(G_N(\langle\Pi\rangle)=G_N(\mathcal{C})\) and so justifies Definition \ref{defn:BSQG}.

\subsection{Actions}\label{subsec:actions}

In this article we quantize the situation of matrices \(M\) acting on tupels of vectors by entrywise left and right multiplication:
\[(v^{(1)},\ldots,v^{(d)})\mapsto(Mv^ {(1)},\ldots,Mv^ {(d)})\quad;\quad (v^{(1)},\ldots,v^{(d)})\mapsto(M^Tv^ {(1)},\ldots,M^Tv^ {(d)}).\]
The case \(d\!=\!1\) is well-known, however, we need it for more general \(d\!\ge\!1\).
\begin{defn}\label{defn:actions_CMQG_CMQS_d-vector_case}
Let \(u_G\!:=\!(u_{ij})_{1\le i,j\le N}\) be a matrix of generators associated to a CMQG \(G\). Let
\[x:=\left(\begin{pmatrix}x_{11}\\ \vdots\\ x_{N1}\end{pmatrix},\ldots,\begin{pmatrix}x_{1d}\\ \vdots\\ x_{Nd}\end{pmatrix}\right)\] 
be a tupel of \(d\) vectors whose entries generate a unital \(C^*\)-algebra \(C(X)\), associated to a compact quantum space \(X\).
\begin{itemize}
\item[(i)]A \emph{left matrix-vector action} \(G\!\curvearrowright\!X\) is a unital \(^*\)-homomorphism \(\alpha\!:C(X)\!\rightarrow\! C(G)\otimes C(X)\) satisfying
\[\alpha(x_{ij})=\displaystyle\sum_{k=1}^{N} u_{ik}\otimes x_{kj}\quad\quad\textnormal{for }1\!\le\!i\!\le\!N,1\!\le\!j\!\le\!d.\]
\item[(ii)]A \emph{right matrix-vector action} \(X\!\curvearrowleft\!G\) is a unital \(^*\)-homomorphism \(\beta\!:\!C(X)\!\rightarrow\! C(G)\otimes C(X)\)  satisfying
\[\beta(x_{ij})=\displaystyle\sum_{k=1}^{N} u_{ki}\otimes x_{kj}\quad\quad\textnormal{for }1\!\le\!i\!\le\!N,1\!\le\!j\!\le\!d.\]
\end{itemize}
\end{defn}
Note that \(\alpha\) and \(\beta\) are nothing but actions on each of the quantum vectors \(\begin{pmatrix}x_{1j}\\ \vdots\\ x_{Nj}\end{pmatrix}\). In other words, for a fixed \(j\), the restriction of \(\alpha\) to \(C^*(x_{1j},\ldots,x_{Nj})\) yields an action of the CMQG on quantum vectors in the well-known sense.
\begin{rem}
The above maps are  special cases of a general left/right action of a compact quantum group \(G\) on a compact quantum space \(X\) (see for example \cite[Definition 1.4]{podles1995}). Those are given by  \(^*\)-homomorphisms \(\tilde{\alpha},\tilde{\beta}:C(X)\!\rightarrow\! C(G)\otimes C(X)\) satisfying
\begin{itemize}
\item[ (a)] \((\Delta\otimes\mathds{1})\circ\tilde{\alpha}=(\mathds{1}\otimes\tilde{\alpha})\circ\tilde{\alpha}\)
\item[ (a')] \((\Sigma\otimes\mathds{1})(\Delta\otimes\mathds{1})\circ\tilde{\beta}=(\mathds{1}\otimes\tilde{\beta})\circ\tilde{\beta}\),\newline
where \(\Sigma\) denotes the flip map given by \(x\!\otimes\!y\mapsto y\!\otimes\!x\),
\item[(b)] \(\big(\tilde{\alpha}\big(C(X)\big)\big(C(G)\otimes\mathds{1}\big)\big)\) and \(\big(\tilde{\beta}\big(C(X)\big)\big(C(G)\otimes\mathds{1}\big)\big)\) are linearly dense in \(C(G)\otimes C(X).\)
\end{itemize}
Considering the maps \(\alpha\) and \(\beta\) from Definition \ref{defn:actions_CMQG_CMQS_d-vector_case}, the proof of property (a) and (a') is a direct consequence of the special form of \(\Delta\) (see Definition \ref{defn:CMQG}) and (b) follows from the invertibility of \(u_G\) and \(\bar{u}_G\).
\end{rem}

\section{Definition of partition quantum spaces}

\subsection{The case of one vector}


We motivate our work by regarding the classical case and the one vector case first.
\begin{ex}\label{ex:S_N}
The symmetric group \(S_N\!\subset\!U_N\!\subset\!M_N(\C)\) is an easy quantum group and it canonically acts on the vector set \(X\!=\!\{e_1,\ldots,e_N\}\!\subseteq\!\C^N\), the standard orthonormal basis of \(\C^N\). We observe that this set of vectors coincides with the set formed by the first columns of all matrices in \(S_N\).
\end{ex}
Translating this observation to the more general setting of an easy quantum group \(G_N(\Pi)\), we should be able to construct quantum spaces inspired by one column in the matrix of generators \(u_{G_N(\Pi)}\). The question therefore is, which structure we have within one column of \(u_{G_N(\Pi)}\).
The relations between the matrix entries \(u_{ij}\) are given by partitions as described in Definition \ref{defn:quantum_group_relations_using_T_i's}.
In general, the resulting equations do not stay within one column, so our first aim is to extract from a given partition certain relations that do so.
\begin{lem}\label{lem:first_column-relations}
Let \(N\!\in\!\N\) and \(p\!\in\!P(\omega,\omega')\!\subseteq\!P(k,l)\) be a partition. Let \(G\) be an easy quantum group such that the relations \(R^{Gr}_p(u_G)\) are fulfilled. 
Using Notation \ref{not:partitioning_of_labelings} we have for all \(1\!\le\!i\!\le\!r\)
\begin{equation}\label{eqn:partition_induced_relations_for_the_u_i1_vector_case}
\sum_{t\in T_i}u_{t_11}^{\omega_1}\dots u_{t_k1}^{\omega_k}=\sum_{t'\in T'_i}u_{t'_11}^{\omega'_1}\dots u_{t'_l1}^{\omega'_l}.
\end{equation}
\end{lem}
\begin{proof}
By Lemma \ref{lem:relations_on_generators_imply_relations_for_category} also the relations \(R_{p^*}^{Gr}(u_G)\) and \(R_{pp^*}^{Gr}(u_G)\) are fulfilled.
Additionally, the decomposition \(T'_0\dot{\cup}\ldots\dot{\cup}T'_r\) of \([N]^l\) for the lower row of \(p\) coincides with the decomposition of \([N]^l\) for both the upper and lower row of \(pp^*\).
Assume that \(\displaystyle(\underbrace{1,\ldots,1}_{k\textnormal{ entries}})\!\in\!T_1\vspace{6pt}\) and \(\displaystyle(\underbrace{1,\ldots,1}_{l\textnormal{ entries}})\!\in\!T'_1\) holds.
The relations \(R^{Gr}_{pp^*}(u_G)\) in particular say (see  Definition \ref{defn:quantum_group_relations_using_T_i's}) that
\[\sum_{t'\in T'_i}u_{t'_11}^{\omega'_1}\dots u_{t'_l1}^{\omega'_l}
=\sum_{t'\in T'_1}u_{\beta'_1t'_1}^{\omega'_1}\dots u_{\beta'_lt'_l}^{\omega'_l}\]
for every \(1\!\le\!i\le\!r\) and \(\beta'\!\in\!T'_i\).
Using this, for any \(\beta'\!\in\!T'_i\) we have
\[
\sum_{t\in T_i}u_{t_11}^{\omega_1}\dots u_{t_k1}^{\omega_k}\;
\overset{R^{Gr}_p(u_G)}{=}\;\sum_{t'\in T'_1}u_{\beta'_1t'_1}^{\omega'_1}\dots u_{\beta'_lt'_l}^{\omega'_l}\\
\;\overset{R^{Gr}_{pp^*}(u_G)}{=}\;\sum_{t'\in T'_i}u_{t'_11}^{\omega'_1}\dots u_{t'_l1}^{\omega'_l}.\vspace{11pt}
\]
Note that this argument is valid even if \(k\!=\!0\) or \(l\!=\!0\) as we then have \(\epsilon\!\in\!T_1\) or \(\epsilon\!\in\!T'_1\), respectively and the corresponding side in Equation \ref{eqn:partition_induced_relations_for_the_u_i1_vector_case} is \(\mathds{1}\) (see Remark \ref{rem:case_k,l=0}).
\end{proof}

We are now ready to formulate the definitions leading to partition quantum spaces. See also the appendix for an overview on all relations related to partitions.

\begin{defn}\label{defn:quantum_space_relations_for_the_x_i}
Let \(N\!\in\!\N\) and \(x\!:=\!(x_{1},\ldots,x_N)^T\) be a vector of generators. Let \(p\!\in\!\mathcal{P}(\omega,\omega')\subseteq \mathcal{P}(k,l)\) be a partition.
Using Notation \ref{not:partitioning_of_labelings}, we associate with \(p\) the following relations on \(x\):
\[R^{Sp}_p(x):\quad\sum_{t\in T_i}x_{t_1}^{\omega_1}\cdots x_{t_k}^{\omega_k}
=\sum_{t'\in T'_i}x_{t'_1}^{\omega'_1}\cdots x_{t'_l}^{\omega'_l}\quad\quad\forall\; 1\!\le\!i\le\!r.\]
Again note that for \(k\!=\!0\) or \(l\!=\!0\) the corresponding side of the equation above is equal to \(\mathds{1}\) (see Remark \ref{rem:case_k,l=0}).
\end{defn}
Consider for example the partition {\paarpartwb} (i.e. \(r\!=\!1\)). The relations \(R^{Sp}_{\!\scalebox{0.7}{\paarpartwb}\!}(x)\) read as
\begin{equation}\label{eqn:quantum_space_relations_for_paarpartwb}
R^{Sp}_{\!\scalebox{0.7}{\paarpartwb}\!}(x):\quad\mathds{1}=\sum_{i=1}^{N}x_{i}x_{i}^*
\end{equation}

\begin{defn}\label{defn:PQS_vector_case}
Let \(N\!\in\!\N\) and \(\Pi\supseteq\{\paarpartwb,\paarpartbw,\baarpartwb,\baarpartbw\}\) be a set of partitions.
Then we define the universal \(C^*\)-algebra 
\[C\big(X_N(\Pi)\big):=C^*(x_{1},\ldots,x_{N}\;|\;\forall p\!\in\!\Pi:\textnormal{ the relations } R^{Sp}_p(x)\textnormal{ hold})\]
and call it the \emph{non-commutative functions} on the  \emph{partition quantum space (PQS) \(X_N(\Pi)\) of one vector}.
\end{defn}
\begin{rem}\label{rem:mcpp_guarantees_existence_of_PQS}
As seen above (Equation \ref{eqn:quantum_space_relations_for_paarpartwb}), the relations \(R^{Sp}_{\!\scalebox{0.7}{\paarpartwb}\!}(x)\) guarantee \(0\!\le x_ix_i^*\!\le\!1\), so the universal \(C^*\)-algebra \(C\big(X_N(\Pi)\big)\) exists.
\end{rem}

\begin{rem}
Note that for a given partition \(p\) we already have another kind of relations: \(R^{Gr}_p(u_G)\), associated to a matrix \(u_G\), see Definition \ref{defn:quantum_group_relations_using_T_i's}. We adress these two definitions as the \emph{quantum group relations} \(R^{Gr}_p(u_G)\) (for the \(u_{ij}\)) and the \emph{quantum space relations} \(R^{Sp}_p(x)\) (for the \(x_{i}\)), respectively.
\end{rem}

As we have seen so far, there are strong similarities between PQSs and easy quantum groups. This is not surprising, as the quantum space relations \(R^{Sp}_p(x)\) are motivated by the quantum group relations \(R^{Gr}_p(u_G)\).
The following theorem makes this observation precise:

\begin{thm}\label{thm:hom_from_PQS_into_BSQG_vector_case}
Let \(N\!\in\!\N\) and \(\Pi\supseteq\{\paarpartwb,\paarpartbw,\baarpartwb,\baarpartbw\}\) be a set of partitions. Let \(X_N(\Pi)\) be the corresponding PQS with vector of generators \(x\!=\!(x_i)\) and \(G_N(\Pi)\) the corresponding easy quantum group with matrix of generators \(u_{G_N(\Pi)}\!=\!(u_{ij})\). Then the mapping
\[\phi:x_i\mapsto u_{i1}\quad\quad,\quad\quad1\!\le\!i\!\le\!N\]
defines a unital \(^*\)-homomorphism from \(C\big(X_N(\Pi)\big)\) to \(C\big(G_N(\Pi)\big)\).
\end{thm}
\begin{proof}
This follows directly from Lemma \ref{lem:first_column-relations}.
\end{proof}

\subsection{The case of \(d\) vectors}\label{subsec:PQS_d-vector_case}

In this section we generalize quantum vectors to \(d\)-tupels of quantum vectors.
We will see already in the commutative case that it is necessary to perform this step. Roughly speaking, we can start with an easy quantum group \(G\) and define a canonical space it acts on, arising from the first column of its fundamental unitary \(u_{G}\), but in general we cannot recover this quantum group as the quantum symmetry group of the constructed space. We then need to consider a tuple of columns.
\begin{ex}\label{ex:S'_N_2}
The easy quantum groups \(H_N\) and \(S_N'\) are both classical groups and arise from \(S_N\) in the following way:
Starting with \(S_N\), we can put a global factor \(\pm 1\) in front of a permutation matrix and end up with the group \(S'_N\).
For matrices in \(H_N\) this additional factor \(\pm1\) is an entrywise choice, so \(H_N\) consists of all matrices with one entry \(\pm1\) in every row and column (and the rest vanishing).
See also \cite{banicaspeicherliberation} for the definitions of \(S'_N\) and \(H_N\).
We clearly have \(S'_N\subseteq H_N\) and inequality for \(N\!>\!1\).

Consider the space of first columns of matrices in \(S'_N\), i.e. the vector set \(X\!=\!\{\pm e_1,\ldots,\pm e_N\}\).
Note, that the first column space of \(H_N\) is \(X\), too.
Taking \(X\) and asking for its symmetry group we do not end up with \(S'_N\) but with \(H_N\).

To overcome this problem, we consider the first \emph{two} columns of \(S'_N\), which yields the set of vector pairs \(X':=\{(\pm(e_i,e_j)\;|\;1\!\le\!i,j\!\le\!N\}\).
We see that \(S'_N\) acts on \(X'\) by entrywise matrix-vector multiplication, but \(H_N\) does not,  as for example \((e_1,-e_2)\) is in the image of \(H_N\) acting on \(X'\) but \((e_1,-e_2)\!\notin\!X'\).
In fact: We may recover the group \(S'_N\) as the (quantum) symmetry group of \(X'\), so the space arising from two columns contains enough informations to recover \(S_N'\).
\end{ex} 
As a first step, we extract analogous to Lemma \ref{lem:first_column-relations} for a given CMQG relations for the entries \(u_{ij}\) of \(u_G\), which are more suitable for our purposes.
\begin{defn}\label{defn:first_d-many_columns-relations}
Let \(p\!\in\!\mathcal{P}(\omega,\omega')\subseteq \mathcal{P}(k,l)\) be  a partition and \(G\)  a CMQG with matrix of generators \(u_G\). Using Notation \ref{not:partitioning_of_labelings} we associate with the partition \(p\) the following relations, denoted by \(R^ {Sp}_p(u_G)\):
\begin{itemize}
\item[(i)]\(\displaystyle\sum_{t\in T_i}u_{t_1\gamma_1}^{\omega_1}\dots u_{t_k\gamma_k}^{\omega_k}=\sum_{t'\in T'_i}u_{t'_1\gamma'_1}^{\omega'_1}\dots u_{t'_l\gamma'_l}^{\omega'_l}\quad,\quad1\!\le\!i,j\!\le\!r,\gamma\!\in\!T_j\textnormal { and }\gamma'\!\in\!T'_j\).
\item[(ii)]\(\displaystyle\sum_{t\in T_i}u_{t_1\gamma_1}^{\omega_1}\dots u_{t_k\gamma_k}^{\omega_k}=0\quad,\quad1\!\le\!i\!\le\!r\textnormal { and }\gamma\!\in\!T_0\).
\item[(iii)]\(\displaystyle\sum_{t'\in T'_i}u_{t'_1\gamma'_1}^{\omega'_1}\dots u_{t'_l\gamma'_l}^{\omega'_l}=0\quad,\quad1\!\le\!i\!\le\!r\textnormal { and }\gamma'\!\in\!T'_0\).
\end{itemize}
\end{defn}
\begin{rem}\label{rem:quantum_space-relations:sums_only_depend_on_T_j_case_u}
Note, that if (i) is fulfilled, then  the left side of equation (i) does not depend on our choice of \(\gamma\!\in\!T_j\) and likewise the right side does not depend on \(\gamma'\!\in\!T'_j\):
\[\sum_{t\in T_i}u_{t_1\gamma_1}^{\omega_1}\dots u_{t_k\gamma_k}^{\omega_k}
=\sum_{t\in T_i}u_{t_1\tilde{\gamma}_1}^{\omega_1}\dots u_{t_k\tilde{\gamma}_k}^{\omega_k}\quad,\quad\forall\gamma,\tilde{\gamma}\!\in\!T_j\]
and
\[\sum_{t'\in T'_i}u_{t'_1\gamma'_1}^{\omega'_1}\dots u_{t'_l\gamma'_l}^{\omega'_l}
=\sum_{t'\in T'_i}u_{t'_1\tilde{\gamma}'_1}^{\omega'_1}\dots u_{t'_l\tilde{\gamma}'_l}^{\omega'_l}\quad,\quad\forall\gamma',\tilde{\gamma}'\!\in\!T'_j.\]
\end{rem}

Recall that in the appendix, all relevant relations are summarized on one page, for the readers convenience.
\begin{lem}\label{lem:quantum_group_relations_imply_quantum_space_relations}
Let \(G\) be a CMQG with matrix of generators \(u_G=\big(u_{ij}\big)_{1\le i,j\le N}\) and \(p\!\in\!P(\omega,\omega')\subseteq P(k,l)\) be a partition. Then it holds
\begin{itemize}
\item [(1)]\(R_p^{Gr}(u_G),\;R_{pp^*}^{Gr}(u_G),\;R_{p^*}^{Gr}(u_G)\quad\Rightarrow\quad R_p^{Sp}(u_G)\),
\item [(2)]\(R_p^{Sp}(u_G),\;R_p^{Sp}(u_G^T)\quad\Rightarrow\quad R_p^{Gr}(u_G)\).
\end{itemize} 
\end{lem}
\begin{proof}
For (1) assume that the quantum group relations of \(p\), \(pp^*\) and \(p^*\) are fulfilled.
Equation (i) in Definition \ref{defn:first_d-many_columns-relations} is proved with the same arguments as in Lemma \ref{lem:first_column-relations}, we just replace the multi indices \((1,\ldots,1)\) by \(\gamma\!\in\!T_j\) and \(\gamma'\!\in\!T'_j\), respectively: For any \(\beta'\!\in\!T'_i\) it holds
\[
\sum_{t\in T_i}u_{t_1\gamma_1}^{\omega_1}\dots u_{t_k\gamma_k}^{\omega_k}
\;\overset{R^{Gr}_p(u_{G})}{=}\;\sum_{t'\in T'_j}u_{\beta'_1t'_1}^{\omega'_1}\dots u_{\beta'_lt'_l}^{\omega'_l}\\
\;\overset{R^{Gr}_{pp^*}(u_{G})}{=}\;\sum_{t'\in T'_i}u_{t'_1\gamma'_1}^{\omega'_1}\dots u_{t'_l\gamma'_l}^{\omega'_l}\\
.\vspace{11pt}
\]
Equations (ii) and (iii) in Definition \ref{defn:first_d-many_columns-relations} follow directly from the relations \(R_p^{Gr}(u_G)\) and \(R_{p^*}^{Gr}(u_G)\), see Definition \ref{defn:quantum_group_relations_using_T_i's}.\newline
For (2) assume that \(R_p^{Sp}(u_G)\) and \(R_p^{Sp}(u_G^T)\) are fulfilled. We have to prove the quantum group relations of the partition \(p\). Note that \(R_p^{Sp}(u_G)\) and \(R_p^{Sp}(u_G^T)\) already include the quantum group relations (ii) and (iii) in Definition \ref{defn:quantum_group_relations_using_T_i's}, so there is only (i) left to prove. As \(R_p^{Sp}(u_G)\) holds, we know for every \(1\!\le\!i,j\!\le\!r\), \(\gamma\!\in\!T_j\) and \(n'\!\in\!T'_j\) that
\[\sum_{t\in T_i}u_{t_1\gamma_1}^{\omega_1}\cdots u_{t_k\gamma_k}^{\omega_k}
=\sum_{m'\in T'_i}u_{m'_1n'_1}^{\omega'_1}\cdots u_{m'_ln'_l}^{\omega'_l}.\]
Out of this we can straightforwardly deduce the desired relation, as for any \(\gamma'\!\in\!T'_i\) the right side can now be written as
\begin{align*}
\begin{split}\label{eqn:thm_X_is_BSQS_if_d=N_computation:for_relation_(i)}
\sum_{m'\in T'_i}u_{m'_1n'_1}^{\omega'_1}\cdots u_{m'_ln'_l}^{\omega'_l}
&\overset{(*)}{=}\frac{1}{|T'_j|}\sum_{t'\in T'_j}\sum_{m'\in T'_i}u_{m'_1t'_1}^{\omega'_1}\cdots u_{m'_lt'_l}^{\omega'_l}\\
&\overset{(*)}{=}\frac{|T'_i|}{|T'_j|}\sum_{t'\in T'_j}u_{\gamma'_1t'_1}^{\omega'_1}\cdots u_{\gamma'_lt'_l}^{\omega'_l}\\
&=\sum_{t'\in T'_j}u_{\gamma'_1t'_1}^{\omega'_1}\cdots u_{\gamma'_lt'_l}^{\omega'_l}\;,
\end{split}
\end{align*}
where we used Remark \ref{rem:quantum_space-relations:sums_only_depend_on_T_j_case_u} in (\(*\)) -- once for \(u_G\) and once for \(u_G^T\) -- and \(|T'_i|=|T'_j|\), see Remark \ref{rem:case_k,l=0} .
\end{proof}

\begin{rem}\label{rem:quantum_space_relations_always_fulfilled_for_BSQGs}
Note that by Lemma \ref{lem:relations_on_generators_imply_relations_for_category} the assumptions of implication (1) in Lemma \ref{lem:quantum_group_relations_imply_quantum_space_relations} are always fulfilled if we consider an easy quantum group \(G_N(\Pi)\) with \(p\!\in\!\langle\Pi\rangle\). As these relations also hold for \(u_{G_N(\Pi)}^T\) we alltogether have
\[R_p^{Gr}(u_{G_N(\Pi)})\Leftrightarrow R_p^{Sp}(u_{G_N(\Pi)}),R_p^{Sp}(u_{G_N(\Pi)}^T).\]
\end{rem}

We use the deduced equations to define relations similar to Definition \ref{defn:quantum_space_relations_for_the_x_i}, but now in the case where \(x\) is a \(d\)-tupel of vectors.

\begin{defn}\label{defn:quantum_space_relations_for_the_x_ij}
Let \(d,N\!\in\!\N\) with \(d\!\le\!N\) and 
\[x\!:=\!(x_{ij})\!:=\!\left(\begin{pmatrix}x_{11}\\ \vdots\\ x_{N1}\end{pmatrix},\ldots,\begin{pmatrix}x_{1d}\\ \vdots\\ x_{Nd}\end{pmatrix}\right)\]
a tupel of vectors of generators \(x_{ij}\).
Let \(p\!\in\!\mathcal{P}(\omega,\omega')\!\subseteq\!\mathcal{P}(k,l)\) be a partition. Using Notation \ref{not:partitioning_of_labelings}, we associate with \(p\) the following relations, denoted by \(R^{Sp}_p(x)\):
\begin{itemize}
\item[(i)]\(\displaystyle\sum_{t\in T_i}x_{t_1\gamma_1}^{\omega_1}\dots x_{t_k\gamma_k}^{\omega_k}=\sum_{t'\in T'_i}x_{t'_1\gamma'_1}^{\omega'_1}\dots x_{t'_l\gamma'_l}^{\omega'_l}\quad,\quad 1\!\le\!i,j\!\le\!r,\; \gamma\!\in\!T_j\!\cap\![d]^k,\;\gamma'\!\in\!T'_j\!\cap\![d]^l\).\vspace{4pt}
\item[(ii)]\(\displaystyle\sum_{t\in T_i}x_{t_1\gamma_1}^{\omega_1}\dots x_{t_k\gamma_k}^{\omega_k}=0\quad,\quad 1\!\le\!i\!\le\!r,\;\gamma\!\in\!T_0\!\cap\![d]^k\).\vspace{4pt}
\item[(iii)]\(\displaystyle\sum_{t'\in T'_i}x_{t'_1\gamma'_1}^{\omega'_1}\dots x_{t'_l\gamma'_l}^{\omega'_l}=0\quad,\quad 1\!\le\!i\!\le\!r,\;\gamma'\!\in\!T'_0\!\cap\![d]^l\).\vspace{8pt}
\end{itemize}
\end{defn}

\begin{rem}
We denoted with \(R_p^{Sp}(\cdot)\) three types of relations: First, note that Definition \ref{defn:quantum_space_relations_for_the_x_i} is a special case of Definition \ref{defn:quantum_space_relations_for_the_x_ij}.
Second, consider the relations \(R_p^{Sp}(u_G)\)  from Definition \ref{defn:first_d-many_columns-relations} and \(R_p^{Sp}(x)\) from Definition \ref{defn:quantum_space_relations_for_the_x_ij}. We can see both situations as special cases of one definition where the argument of \(R_p^{Sp}(\cdot)\) is a collection of indexed symbols \((y_{ij})\) with \(i\!\times\!j\!\in\![N]\times[d]\) and the resulting relations in particular depend on the given \(d\). See also the appendix for an overview on these relations.
\end{rem}
\begin{rem}\label{rem:quantum_space-relations:sums_only_depend_on_T_j_case_x}
If the corresponding relations are fulfilled, we have again (compare Remark \ref{rem:quantum_space-relations:sums_only_depend_on_T_j_case_u}) that the sums appearing in Definition \ref{defn:quantum_space_relations_for_the_x_ij} do not depend on the chosen \(\gamma\!\in\!T_j\cap[d]^k\) and  \(\gamma'\!\in\!T'_j\cap[d]^l\), respectively:
\[\sum_{t\in T_i}x_{t_1\gamma_1}^{\omega_1}\dots x_{t_k\gamma_k}^{\omega_k}
=\sum_{t\in T_i}x_{t_1\tilde{\gamma}_1}^{\omega_1}\dots x_{t_k\tilde{\gamma}_k}^{\omega_k}\quad,\quad\forall \gamma,\tilde{\gamma}\!\in\!T_j\cap[d]^k.\]
and
\[\sum_{t'\in T'_i}x_{t'_1\gamma'_1}^{\omega'1}\dots x_{t'_l\gamma'_l}^{\omega'_l}
=\sum_{t'\in T'_i}x_{t'_1\tilde{\gamma}'_1}^{\omega'_1}\dots x_{t'_l\tilde{\gamma}'_l}^{\omega'_l}\quad,\quad\forall \gamma',\tilde{\gamma}'\!\in\!T'_j\cap[d]^l\]
\end{rem}
\begin{ex}
Consider any \(d\)-tupel \(x\) in the sense above and again the partition
\[p=
\setlength{\unitlength}{0.3cm}
\begin{picture}(5,2.4)
\put(0,-1.8) {$\bullet$}
 \put(1,-1.8) {$\bullet$}
 \put(2,-1.8) {$\circ$}
 \put(0,1.6) {$\circ$}
 \put(1,1.6) {$\circ$}
 \put(2,1.6) {$\bullet$}
\put(-0.11,2.6){\partii{1}{0}{1}}
\put(-0.12,2.6){\parti{2}{2}}
\put(-0.12,-1){\uppartii{1}{1}{2}}
\put(-0.12,-1){\upparti{1}{0}}
\end{picture}\hspace{-0.5cm}
.\vspace{11pt}\]
Recall the decomposition of labelings induced by \(p\) into sets \(T_i\) and \(T'_i\) as in Example \ref{ex:decomposition_of_labelings}.
In virtue of the cases (i)-(iii) in Definition \ref{defn:quantum_space_relations_for_the_x_ij} the relations \(R^{Sp}_p(x)\) read as:
\begin{itemize}
\item[(i)]
\(\displaystyle\left(\sum_{t=1}^{N}x_{tj_1}x_{tj_1}\right)x_{ij_2}^*
=\left(\sum_{t'=1}^{N}x_{t'j_3}^*\right)x_{ij_2}^*x_{ij_2}\quad\quad\forall i\in[N], j_1,j_2,j_3\in[d]\vspace{6pt}\)
\item[(ii)]
\(\displaystyle \left(\sum_{t=1}^{N}x_{tj_1}x_{tj_2}\right)x_{ij_3}^*
=0\quad\quad\forall i\in[N], j_1,j_2,j_3\in[d], j_1\!\neq\!j_2\vspace{6pt}\)
\item[(iii)]\(\displaystyle \left(\sum_{t'=1}^{N}x_{t'j_1}^*\right)x_{ij_2}^*x_{ij_3}\!=\!0\quad\quad\forall i\in[N], j_1,j_2,j_3\in[d], j_2\!\neq\!j_3\)

\end{itemize}
\end{ex}
\begin{ex}
To see how the number \(d\) of vectors effects the relations \(R_p^{Sp}(x)\), consider the partition \(p\!=\!\paarpartwb\). Whilst for \(d\!=\!1\) it only holds 
\[\sum_{t'=1}^{N}x_{t'1}x_{t'1}^*\!=\!\mathds{1}\]
we have for \(d\!\ge\!2\)
\[\sum_{t'=1}^{N}x_{t'j_1}x_{t'j_2}^*\!=\!\delta_{j_1j_2}\quad\quad,\quad\quad\forall j_1,j_2\in [d].\]
\end{ex}
\begin{defn}\label{defn:PQS_matrix_case}
Let \(d,N\!\in\!\N\) with \(d\!\le\!N\) and \(x\!:=\!(x_{ij})\) be a \(d\)-tupel of vectors of gene\-rators. Let \(\Pi\supseteq\{\paarpartwb,\paarpartbw,\baarpartwb,\baarpartbw\}\) be a set of partitions. Then we define the universal \(C^*\)-algebra
\[C\big(X_{N,d}(\Pi)\big):=C^*(x_{11},\ldots,x_{Nd}\;|\;\forall p\!\in\!\Pi:\textnormal{ the relations } R^{Sp}_p(x)\textnormal{ hold})\]
and call it the \emph{non-commutative functions} on the \emph{partition quantum space (PQS) \(X_{N,d}(\Pi)\) of \(d\) vectors}.
\end{defn}

\begin{thm}\label{thm:hom_from_PQS_into_BSQG_d-vector_case}
Let \(d,N\!\in\N\), \(d\!\le\!N\) and \(\Pi\supseteq\{\paarpartwb,\paarpartbw,\baarpartwb,\baarpartbw\}\) be a set of partitions. Then the mapping 
\[\phi:x_{ij}\mapsto u_{ij}\quad,\quad1\!\le\!i\!\le\!N,\;1\!\le\!j\!\le\!d\]
defines a \(^*\)-homomorphism from \(C\big(X_{N,d}(\Pi)\big)\) to \(C\big(G_N(\Pi)\big)\).
\end{thm}
\begin{proof}
This follows directly from Lemma \ref{lem:quantum_group_relations_imply_quantum_space_relations} and Remark \ref{rem:quantum_space_relations_always_fulfilled_for_BSQGs}.
\end{proof}

\begin{rem}
For every easy quantum group permutations on rows and/or columns as well as the mapping \(u_{ij}\mapsto u_{ji}\) define \(^*\)-isomorphisms on it. Therefore we actually have a lot of freedom where to map each \(x_{ij}\) to. Namely for every two permutations \(\sigma_1,\sigma_2\in S_N\) we have that
\begin{align*}
\phi_1\,&:\, x_{ij}\mapsto u_{\sigma_1(i)\sigma_2(j)}\\
\phi_2\,&:\, x_{ij}\mapsto u_{\sigma_2(j)\sigma_1(i)}
\end{align*}
both define alternatives to the \(^*\)-homomorphism \(\phi\) in Theorem \ref{thm:hom_from_PQS_into_BSQG_d-vector_case}.
\end{rem}

\begin{rem}
We could define PQSs without requiring \(\{\paarpartwb,\paarpartbw, \baarpartwb,\baarpartbw\}\!\subseteq\!\Pi\). However the existence of  \(C\big(X_{N,d}(\Pi)\big)\) would not be guaranteed, see Remark \ref{rem:mcpp_guarantees_existence_of_PQS}. In any case, for our purposes we only need to consider PQSs with \(\{\paarpartwb,\paarpartbw, \baarpartwb,\baarpartbw\}\!\subseteq\!\Pi\). See \cite{weber_partition_cstar-algebras_I} for more on issues related to \(\{\paarpartwb,\paarpartbw,\baarpartwb,\baarpartbw\}\!\nsubseteq\!\Pi\).\newline
In principle, also \(d\!>\!N\) is possible, but as some of our results only hold for \(d\!\le\!N\), we excluded all other situations in this work.
\end{rem}

\section{Easy quantum groups act on partition quantum spaces}

For integers \(d\!\le\!N\) and a set of partitions \(\Pi\supseteq\{\paarpartwb,\paarpartbw,\baarpartwb,\baarpartbw\}\) we prove in this section that the corresponding easy quantum group always acts on the corresponding PQS of \(d\) vectors. In this sense we can answer Question \ref{quest:does_BSQG_act_on_PQS-introduction} positively and in full generality.
Loosely speaking, every easy quantum group acts (by left and right matrix multiplication) on a quantum space inspired by its first \(d\) columns.

\begin{thm}\label{thm:existence_of_alpha_and_beta}
Let \(d,N\!\in\!\N\), \(d\!\le\!N\) and \(\Pi\!\supseteq\!\{\paarpartwb,\paarpartbw,\baarpartwb,\baarpartbw\}\) be a set of partitions.
Let \(u_G\!:=\!(u_{ij})\) be the matrix of generators associated to the easy quantum group \(G_N(\Pi)\) and \(x:=(x_{ij})\) the \(d\) vectors of generators associated to the PQS \(X_{N,d}(\Pi)\). In the sense of Definition \ref{defn:actions_CMQG_CMQS_d-vector_case} the following holds:
\begin{itemize}
\item[(i)]The mapping
\[\alpha:x_{ij}\mapsto\displaystyle\sum_{k=1}^{N} u_{ik}\otimes x_{kj}\quad\quad\textnormal{for }1\!\le\!i\!\le\!n,\;1\!\le\!j\!\le\!d\]
defines a left matrix-vector action \(G\!\curvearrowright\!X\).
\item[(ii)]The mapping
\[\beta(x_{ij})=\displaystyle\sum_{k=1}^{N} u_{ki}\otimes x_{kj}\quad\quad\textnormal{for }1\!\le\!i\!\le\!n,\;1\!\le\!j\!\le\!d\]
defines a right matrix-vector action \(X\!\curvearrowleft\!G\).
\end{itemize}
\end{thm}

\begin{proof}
We have to prove that \(\alpha\) and \(\beta\) are well-defined \(^*\)-homomorphisms.
We only consider \(\alpha\), the proof for \(\beta\) can be done analogously.
By the universal property of \(C\big(X_{N,d}(\Pi)\big)\) it suffices to prove that the relations (i)-(iii) from Definition \ref{defn:quantum_space_relations_for_the_x_ij} are fulfilled for \(\tilde{x}_{ij}:=\alpha(x_{ij})=\sum_{s=1}^{N}u_{is}\!\otimes\!x_{sj}\).

We start with relation (i) from Definition \ref{defn:quantum_space_relations_for_the_x_ij}.
We use Notation \ref{not:partitioning_of_labelings} and fix \(1\!\le\!i,j\!\le\!r\), \(\gamma\!\in\!T_j\cap[d]^k\) and \(\gamma'\!\in\!T'_j\cap[d]^l\). By definition we have for all \(1\!\le\!m\!\le\!r\)
\begin{equation}\label{eqn:proof_of_existence_of_alpha_and_beta_1}
\sum_{s\in T_m}x_{s_1\gamma_1}^{\omega_1}\cdots x_{s_k\gamma_k}^{\omega_k}
=\sum_{s'\in T'_m}x_{s'_1\gamma'_1}^{\omega'_1}\cdots x_{s'_l\gamma'_l}^{\omega'_l}
\end{equation}
and due to Remarks \ref{rem:quantum_space_relations_always_fulfilled_for_BSQGs} and \ref{rem:quantum_space-relations:sums_only_depend_on_T_j_case_u}  we have that
\begin{equation}\label{eqn:proof_of_existence_of_alpha_and_beta_2}
c_m:=\sum_{t\in T_i}u_{t_1s_1}^{\omega_1}\cdots u_{t_ks_k}^{\omega_k}=\sum_{t'\in T'_i}u_{t'_1s'_1}^{\omega'_1}\cdots u_{t'_ls'_l}^{\omega'_l}
\end{equation}
only depends on \(m\!\in\!\{0,1,\ldots,r\}\) but not on \((s,s')\!\in\!T_m\!\times\!T'_m\). In particular, \(c_0=0\).

Using Equations \ref{eqn:proof_of_existence_of_alpha_and_beta_1} and \ref{eqn:proof_of_existence_of_alpha_and_beta_2} we infer:

{\allowdisplaybreaks
\begin{align*}
\sum_{t\in T_i}\left(\tilde{x}_{t_1\gamma_1}\right)^{\omega_1}\cdots \left(\tilde{x}_{t_k\gamma_k}\right)^{\omega_k}
=&\sum_{s\in[N]^k}\sum_{t\in T_i} u_{t_1s_1}^{\omega_1}\cdots u_{t_ks_k}^{\omega_k}\otimes x_{s_1\gamma_1}^{\omega_1}\cdots x_{s_k\gamma_k}^{\omega_k}\\
=&\sum_{m=0}^{r}\sum_{s\in T_m}\left(\sum_{t\in T_i} u_{t_1s_1}^{\omega_1}\cdots u_{t_ks_k}^{\omega_k}\right)\otimes x_{s_1\gamma_1}^{\omega_1}\cdots x_{s_k\gamma_k}^{\omega_k}\\
=&0+\sum_{m=1}^{r}\sum_{s\in T_m} c_m\otimes x_{s_1\gamma_1}^{\omega_1}\cdots x_{s_k\gamma_k}^{\omega_k}\\
=&0+\sum_{m=1}^{r} c_m\otimes \left(\sum_{s\in T_m}x_{s_1\gamma_1}^{\omega_1}\cdots x_{s_k\gamma_k}^{\omega_k}\right)\\
=&0+\sum_{m=1}^{r}c_m\otimes \left(\sum_{s'\in T'_m}x_{s'_1\gamma'_1}^{\omega'_1}\cdots x_{s'_l\gamma'_l}^{\omega'_l}\right)\\
=&0+\sum_{m=1}^{r}\sum_{s'\in T'_m}c_m\otimes x_{s'_1\gamma'_1}^{\omega'_1}\cdots x_{s'_l\gamma'_l}^{\omega'_l}\\
=&0+\sum_{m=1}^{r}\sum_{s'\in T'_m}\sum_{t'\in T'_i} u_{t'_1s'_1}^{\omega'_1}\cdots u_{t'_ls'_l}^{\omega'_l}\otimes x_{s'_1\gamma'_1}^{\omega'_1}\cdots x_{s'_l\gamma'_l}^{\omega'_l}\\
=&\sum_{m=0}^{r}\sum_{s'\in T'_m}\sum_{t'\in T'_i} u_{t'_1s'_1}^{\omega'_1}\cdots u_{t'_ls'_l}^{\omega'_l}\otimes x_{s'_1\gamma'_1}^{\omega'_1}\cdots x_{s'_l\gamma'_l}^{\omega'_l}\\
=&\sum_{t\in T'_i}\left(\tilde{x}_{t'_1\gamma'_1}\right)^{\omega'_1}\cdots \left(\tilde{x}_{t'_l\gamma'_l}\right)^{\omega'_l}.
\end{align*}
}
For relation (ii) in Definition \ref{defn:quantum_space_relations_for_the_x_ij} we similarly compute for all \(\gamma\!\in\!T_0\)
\begin{align*}
\sum_{t\in T_i}\left(\tilde{x}_{t_1\gamma_1}\right)^{\omega_1}\cdots \left(\tilde{x}_{t_k\gamma_k}\right)^{\omega_k}
=0+\sum_{m=1}^{r}c_m\otimes \underbrace{\sum_{s\in T_m} x_{s_1\gamma_1}^{\omega_1}\cdots x_{s_k\gamma_k}^{\omega_k}}_{=0}=0.
\end{align*}

Analogously we prove (iii): For all \(\gamma'\!\in\!T'_0\) we have
\begin{align*}
\sum_{t\in T'_i}\left(\tilde{x}_{t'_1\gamma'_1}\right)^{\omega'_1}\cdots \left(\tilde{x}_{t'_l\gamma'_l}\right)^{\omega'_l}
=0+\sum_{m=1}^{r}c_m\otimes \underbrace{\sum_{s'\in T'_m} x_{s'_1\gamma'_1}^{\omega'_1}\cdots x_{s'_l\gamma'_l}^{\omega'_l}}_{=0}=0.
\end{align*}
\end{proof}

\section{Quantum symmetry groups of partition quantum spaces}

In the following we want to take a closer look at the connection between a partition quantum space and its quantum symmetry group. We start with a precise definition of quantum symmetry groups in our setting.

\begin{defn}\label{defn:QSymG}
Let \(X_{N,d}(\Pi)\) be a PQS of \(d\) vectors. We call a CMQG \(G\) the \emph{quantum symmetry group of \(X_{N,d}(\Pi)\)} if there are left and right matrix-vector actions \(\alpha\) and \(\beta\) of \(G\) on \(X_{N,d}(\Pi)\) in the sense of Definition \ref{defn:actions_CMQG_CMQS_d-vector_case} and if \(G\) is maximal with this property. I.e. all \(G'\) fulfilling the above satisfy \(G'\!\subseteq\! G\).
\end{defn}
Note, that \(G'\!\subseteq\! G\) means that we have a \(^*\)-homomorphism \(C(G)\rightarrow C(G')\), sending the entries of \(u_G\) canonically to the entries of \(u_{G'}\).
Of course Definition \ref{defn:QSymG} makes sense not only for PQSs, but for every quantum space of vectors \(X\).
\begin{notation}\label{not:fixing_notations}
For the rest of this work we consider the following situation/notation:
Let \(d,N\!\in\!\N\) and \(d\!\le\!N\).
Let \(\Pi\) be a set of partitions containing the set of all mixed coloured pair partitions \(\{\paarpartwb,\paarpartbw,\baarpartwb,\baarpartbw\}\).
Let \(G\) be the  quantum symmetry group of \(X_{N,d}(\Pi)\) with associated matrix of generators \(v_G:=(v_{ij})\).
Let \(u_{G_N(\Pi)}=(u_{ij})\) be the matrix of generators associated to the easy quantum group \(G_N(\Pi)\).
For a fixed partition \(p\!\in\!P(\omega,\omega')\!\subseteq\!P(k,l)\) we always use the decompositions \([N]^k=T_0\,\dot{\cup}\,\ldots\,\dot{\cup}\, T_r\) and \([N]^l=T'_0\,\dot{\cup}\,\ldots\,\dot{\cup}\, T'_r\) as in Notation \ref{not:partitioning_of_labelings}.
\end{notation}

As \(G\) is the quantum symmetry group of \(X_{N,d}(\Pi)\), we know that \(^*\)-homomorphisms \(\alpha\) and \(\beta\) as decribed in Theorem \ref{thm:existence_of_alpha_and_beta} exist.
The question is now: What does this tell us about \(C(G)\) and how may we explicitly compute relations holding in \(C(G)\)?

\begin{ex}\label{ex:how_to_deduce_relations_for_the_v_ij}
Consider \(\Pi=\{\paarpartwb,\paarpartbw,\baarpartwb,\baarpartbw\}\).
In the PQS \(X_{N,1}(\Pi)\) we have for example \(R_{\scalebox{0.7}{\(\baarpartwb\)}}^{Sp}(x)\), i.e. \(\displaystyle\sum_{i}x_{i1}x_{i1}^*=\mathds{1}\). Applying \(\alpha\) to this equation gives
\begin{equation}\label{eqn:alpha_applied_to_x_ij-relation}
\sum_{i=1}^{N}\sum_{k_1,k_2=1}^{N}v_{ik_1}v_{ik_2}^*\otimes x_{k_11}x_{k_21}^*=\mathds{1}.
\end{equation}
This encodes some information about the \(v_{ij}\),  but it is not always easy to read, as the \(C^*\)-algebra \(C\big(X_{N,d}(\Pi)\big)\) might be quite complicated.

\begin{notation}\label{not:evalutationmap}
Consider any permutation matrix \(\sigma\!\in\!S_N\) and the following chain of unital \(^*\)-homomorphisms:
\begin{center}
\begin{tabular}{rccccccc}
\(\textnormal{ev}_{\sigma}\):&\(C\big(X_{N,d}(\Pi)\big)\)&\(\overset{\phi_1}{\longrightarrow} \)&\(C\big(G_N(\Pi)\big)\)&\(\overset{\phi_2}{\longrightarrow}\)& \(C\big(S_N\big)\)&\(\overset{\phi_3}{\longrightarrow}\)&\(\C\)\\
&\(x_{ij}\)&\(\longmapsto\) & \(u_{ij}\)&\(\longmapsto\) &\(\tilde{u}_{ij}\)&\(\longmapsto\)& \(\delta_{i\sigma(j)}\)
\end{tabular}
\end{center}
Here, \((\tilde{u}_{ij})\!=\!u_{S_N}\) is the matrix of generators corresponding to \(C(S_N)\), i.e. 
\[\tilde{u}_{ij}:\;C\big(S_N\big)\rightarrow\C;\;\tilde{\sigma}\mapsto \tilde{u}_{ij}(\tilde{\sigma})\!=\!\tilde{\sigma}_{ij}\!=\!\delta_{i\tilde{\sigma}(j)}\]
is the coordinate function for the entry \((i,j)\).
Recall that the \(\tilde{u}_{ij}\) are projections summing up to one in each row and column, i.e. \(\sum_{k=1}^{N}\tilde{u}_{ik}\!=\! \sum_{k=1}^{N}\tilde{u}_{ki}\!=\!\mathds{1}\) for all \(i\!\in\![N]\).
The existence of \(\phi_1\) is by Theorem \ref{thm:hom_from_PQS_into_BSQG_d-vector_case}.
The map \(\phi_2\) exists as described as we have \(S_N\!\subseteq\!G_N(\Pi)\) for every easy quantum group \(G_N(\Pi)\), see \cite{banicaspeicherliberation,weber_LNM,weber_PMS}.
For \(\phi_3\) observe, that the point evaluation \(f\mapsto f(\sigma)\) is a character on \(C(S_N)\).
\end{notation}

Applying \(\mathds{1}\!\otimes\!\textnormal{ev}_{\sigma}\) for some \(\sigma\) with \(\sigma(1)\!=\!k\) to Equation \ref{eqn:alpha_applied_to_x_ij-relation} results in
\[\mathds{1}=\sum_{k_1,k_2=1}^{N}\sum_{i=1}^{N}v_{ik}v_{ik}^*\otimes \delta_{k_1\sigma(1)}\delta_{k_2\sigma(1)}=\sum_{i=1}^{N}v_{ik}v_{ik}^*\otimes 1.\]
As \(k\!\in\![N]\) was arbitrary, we therefore proved the relations
\[\sum_{i}v_{ik}v_{ik}^*=\mathds{1}\quad\quad\forall k.\]

\end{ex}

With the strategy presented above, we can prove (see Theorem \ref{thm:hom_from_PQS_to_QSymG}) that the relations \(R^{Sp}_p(x)\) hold for any choice of \(d\) columns of \((v_{ij})\).
We first present a preparing result, keeping the consecutive theorem and its proof compact.

\begin{lem}\label{lem:lemma2_relations_for_v_ij}
In the situation of Notation \ref{not:fixing_notations}, for any \(p\!\in\!\Pi\) the relations \(R_p^{Sp}\) of Definition \ref{defn:quantum_space_relations_for_the_x_ij} hold for the first \(d\) columns of \(v\), i.e.:
\begin{itemize}
\item[(i)]
\(\displaystyle\sum_{t\in T_i} v_{t_1\gamma_1}^{\omega_1}\cdots v_{t_k\gamma_k}^{\omega_k}=\sum_{t'\in T'_i} v_{t'_1\gamma'_1}^{\omega'_1}\cdots v_{t'_l\gamma'_l}^{\omega'_l}\quad,\quad1\!\le \!i,j\!\le\!r\), \(\gamma\!\in\!T_{j}\cap[d]^k,\gamma'\!\in\!T'_{j}\cap[d]^ l.\)\vspace{4pt}\newline

\item[(ii)]
\(\displaystyle\sum_{t\in T_i}v_{t_1\gamma_1}^{\omega_1}\cdots v_{t_k\gamma_k}^{\omega_k}=0\quad,\quad1\!\le\!i\!\le\!r,\gamma\!\in\!T_0\cap[d]^k.\)\vspace{8pt}

\item[(iii)]
\(\displaystyle \sum_{t'\in T'_i}v_{t'_1\gamma'_1}^{\omega'_1}\cdots v_{t'_l\gamma'_l}^{\omega'_l}=0\quad,\quad1\!\le\!i\!\le\!r,\gamma\!\in\!T'_0\cap[d]^l.\)\vspace{8pt}
\end{itemize}
\end{lem}
\begin{proof}
We start with the relations in (i).
In virtue of Definition \ref{defn:quantum_space_relations_for_the_x_ij} we have for \(i,j,\gamma,\gamma'\) as defined above
\[\sum_{t\in T_i}x_{t_1\gamma_1}^{\omega_1}\cdots x_{t_k\gamma_k}^{\omega_k}=\sum_{t'\in T'_i}x_{t'_1\gamma'_1}^{\omega'_1}\cdots x_{t'_l\gamma'_l}^{\omega'_l}.\]
We consider \(\sigma\!\in\!S_N\) and apply \((\mathds{1}\otimes\textnormal{ev}_{\sigma})\circ\alpha\) to receive
\begin{align}\label{eqn:lemma2_results_for_v_ij_relation (i)}
\begin{split}
\sum_{t\in T_i} v_{t_1\sigma(\gamma_1)}^{\omega_1}\cdots v_{t_k\sigma(\gamma_k)}^{\omega_k}\otimes 1
=\sum_{t'\in T'_i} v_{t'_1\sigma(\gamma'_1)}^{\omega'_1}\cdots v_{t'_l\sigma(\gamma'_l)}^{\omega'_l}\otimes 1.
\end{split}
\end{align}
For \(\sigma\!=\!\textnormal{id}\) this is our claim.

We pass now to the relations (ii) and (iii). Starting with
\[\sum_{t\in T_i}x_{t_1\gamma_1}^{\omega_1}\cdots x_{t_k\gamma_k}^{\omega_k}=0\quad\quad\textnormal{and}\quad\quad\sum_{t'\in T'_i}x_{t'_1\gamma'_1}^{\omega'_1}\cdots x_{t'_l\gamma'_l}^{\omega'_l}=0\]
we deduce similarly to the case of Equation (i) with the help of \((\mathds{1}\otimes\textnormal{ev}_{\sigma})\circ\alpha\):
\begin{equation}\label{eqn:lemma2_results_for_v_ij_relation_(ii)}
\sum_{t\in T_i} v_{t_1\sigma(\gamma_1)}^{\omega_1}\cdots v_{t_k\sigma(\gamma_k)}^{\omega_k}\otimes 1=0
\quad\quad\textnormal{and}\quad\quad
\sum_{t'\in T'_i}x_{t_1\sigma(\gamma'_1)}^{\omega'_1}\cdots x_{t'_l\sigma(\gamma'_l)}^{\omega'_l}\otimes 1=0
\end{equation}
In both cases this includes with (\(\sigma=\textnormal{id}\)) the desired relation.
\end{proof}

Without any further work we can now prove the existence of \(^*\)-homomorphisms from \(C\big(X_{N,d}(\Pi)\big)\) to \(C(G)\), where \(G\) is the quantum symmetry group of \(X_{N,d}(\Pi)\).

\begin{thm}\label{thm:hom_from_PQS_to_QSymG}
Consider the situation of Notation \ref{not:fixing_notations}. For any \(\sigma_1,\sigma_2\!\in\!S_N\) the following mappings define \(^*\)-homomorphisms from \(C\big(X_{N,d}(\Pi)\big)\) to \(C(G)\):
\[
\phi:\quad x_{ij}\mapsto v_{\sigma_1(i)\sigma_2(j)}
\quad\quad;\quad\quad
\phi^T:\quad x_{ij}\mapsto v_{\sigma_2(j)\sigma_1(i)}\]
In particular (for \(\sigma_1\!=\sigma_2\!=\!\textnormal{id}\)) the mapping
\[\phi:\quad x_{ij}\mapsto v_{ij}\]
defines a \(^*\)-homomorphism from \(C\big(X_{N,d}(\Pi)\big)\) to \(C(G)\).
\end{thm}
\begin{proof}
By the universal property of \(C\big(X_{N,d}(\Pi)\big)\) we only have to show that the relations \(\big(R^{Sp}_p(x)\big)_{p\in\Pi}\) are fulfilled when we replace every \(x_{ij}\) by \(v_{\sigma_1(i)\sigma_2(j)}\) or \(v_{\sigma_2(j)\sigma_1(i)}\), respectively. Thanks to our results so far there is not much left to do.

Consider first the map \(\phi\).
If \(\sigma_1=\textnormal{id}\), then \(\phi\) exists, because we have with Equations \ref{eqn:lemma2_results_for_v_ij_relation (i)} and \ref{eqn:lemma2_results_for_v_ij_relation_(ii)} from Lemma \ref{lem:lemma2_relations_for_v_ij} all required relations for the \(v_{ij}\) at hand.
We can additionally have \(\sigma_1\!\neq\textnormal{id}\) as the mapping \(x_{ij}\mapsto x_{\sigma_1(i)j}\) defines a \(^*\)-isomorphism on \(C\big(X_{N,d}(\Pi)\big)\): Considering the quantum space relations \(R_p^{Sp}(x)\) from Definition \ref{defn:quantum_space_relations_for_the_x_ij}, we see that applying \(i\mapsto\sigma(i)\) only permutes summands.

The existence of \(\phi^T\) is proved by starting again at Lemma \ref{lem:lemma2_relations_for_v_ij} and replacing \(\alpha\) with \(\beta\). We obtain all the ensuing results with  \(v_{ij}\) replaced by \(v_{ji}\).
\end{proof}

\begin{rem}\label{rem:new_remark_to_lemma2_relations_for_v_ij}
Having a closer look at Equation \ref{eqn:lemma2_results_for_v_ij_relation (i)} one can even deduce that for arbitrary \(\sigma,\sigma'\!\in\!S_N\) we have
\[\sum_{t\in T_i} v_{t_1\sigma(\gamma_1)}^{\omega_1}\cdots v_{t_k\sigma(\gamma_k)}^{\omega_k}
=\sum_{t'\in T'_i} v_{t'_1\sigma'(\gamma'_1)}^{\omega'_1}\cdots v_{t'_l\sigma'(\gamma'_l)}^{\omega'_l}\]
as long as \(\sigma\) and \( \sigma'\) coincide on the through-block labelings of \(\gamma\) (or \(\gamma'\), which is the same condition).
In this sense we can ignore in Lemma \ref{lem:lemma2_relations_for_v_ij}, Equation (i) the restrictions \(\gamma\!\in\![d]^k\) and \(\gamma'\!\in\![d]^l\) as long as each tupel has at most \(d\) different entries.
For the relations (ii) and (iii) this follows directly from the Equations \ref{eqn:lemma2_results_for_v_ij_relation_(ii)}.
\end{rem}

For the case \(d\!=\!N\) we can now fully answer Question \ref{quest:is_BSQG_QSymG_of_PQS_introduction} from the introduction.

\begin{cor}\label{cor:X_is_BSQS_if_d=N}
Consider the situation of Notation \ref{not:fixing_notations}. If \(d\!=\!N\), then \(G=G_N(\Pi)\), i.e. \(G_N(\Pi)\) is the quantum symmetry group of \(X_{N,N}(\Pi)\).
\end{cor}
\begin{proof}
Due to Theorem \ref{thm:existence_of_alpha_and_beta} we already know \(G_N(\Pi)\!\subseteq\!G\) so there is only ``\(\supseteq\)'' left to prove, i.e. it remains to show that the quantum group relations \(\big(R_p^{Gr}(v_G)\big)_{p\in\Pi}\) are fulfilled.
In the case \(d\!=\!N\) Theorem \ref{thm:hom_from_PQS_to_QSymG} reads as \(R_p^ {Sp}(x)\!\Rightarrow\!R_p^{Sp}(v_G),R_p^{Sp}(v_G^T)\) and by part (2) of Lemma \ref{lem:quantum_group_relations_imply_quantum_space_relations} the quantum space relations for \(v_G\) and \(v_G^T\) imply \(R_p^{Gr}(v_G)\).
\end{proof}

There is one further consequence of the theorem above, coming from the fact, that \(G_N(\Pi)\) only depends on the category \(\langle\Pi\rangle\) and not \(\Pi\) itself.

\begin{cor}\label{cor:QSymG_only_depends_on_category_if_d=N}
In the situation of Notation \ref{not:fixing_notations} the quantum symmetry group of \(X_{N,N}(\Pi)\) only depends on \(\langle\Pi\rangle\), not \(\Pi\) itself.
\end{cor}

\begin{rem}\label{rem:proof-plan_for_d<N}
Having a closer look at the proof of Corollary \ref{cor:X_is_BSQS_if_d=N}, we see that it is just an application of the relations proved in Lemma \ref{lem:lemma2_relations_for_v_ij}.
Of course it heavily uses the fact \(d\!=\!N\) in the way that \(\gamma\!\in\![d]^k\) and \(\gamma'\!\in\![d]^l\) are actually no further restrictions.
In the following we will consider the situation as in Corollary \ref{cor:X_is_BSQS_if_d=N} but with \(d\!<\!N\). In the sense of Lemma \ref{lem:lemma2_relations_for_v_ij} and Remark \ref{rem:new_remark_to_lemma2_relations_for_v_ij} the only thing to prove is that the equations there are fulfilled for \(\gamma\!\in\!T_j\) and \(\gamma'\!\in\!T'_j\), potentially each with more than \(d\) different entries. Up to now, we are only able to do this for concretely given sets \(\Pi\), so further results will depend not only on \(\langle\Pi\rangle\), but \(\Pi\) itself.
\end{rem}

\section{The Free case}\label{sec:free_case}
In this section we assume that \(\Pi\) defines a free easy quantum group, i.e. \(\Pi\) is a set of non-crossing partitions. In the sense of Remark \ref{rem:proof-plan_for_d<N} we want to find situations where \(G_N(\Pi)\) is the quantum symmetry group of \(X_{N,d}(\Pi)\) for some \(d<N\). Note that this implies the same result for \(d'\) with \(d\!\le\!d'\!\le\!N\), so (for fixed \(\Pi\)) we want to show this for \(d\) as small as possible.\vspace{11pt}

We first gather some results about relations on the \(v_{ij}\) of Notation \ref{not:fixing_notations} which are implied by special partitions \(p\!\in\!\Pi\) (and special choices of \(\Pi\)).

\begin{lem}\label{lem:relations_of_four-block}
Consider the situation as in Notation \ref{not:fixing_notations} with \(d\!\le\!N\) arbitrary. Assume \(\vierpartwbwb\!\in\!\Pi\). Then the relations \(R^{Gr}_{\scalebox{0.8}{\vierpartwbwb}}(v_G)\)\vspace{3pt}   are fulfilled, i.e.
\[\sum_{t'=1}^{N}v_{\gamma'_1t'}v_{\gamma'_2t'}^*v_{\gamma'_3t'}v_{\gamma'_4t'}^*=\delta_{\gamma'_1\gamma'_2}
\delta_{\gamma'_2\gamma'_3}\delta_{\gamma'_3\gamma'_4}\quad\quad\forall\;\gamma'\!\in\![N]^4.\]
The analogous result holds for \(\vierpartwwbb\!\in\!\Pi\) and the relations \(R^{Gr}_{\scalebox{0.8}{\vierpartwwbb}}(v_G)\), i.e.
\[\sum_{t'=1}^{N}v_{\gamma'_1t'}v_{\gamma'_2t'}v_{\gamma'_3t'}^*v_{\gamma'_4t'}^*=\delta_{\gamma'_1\gamma'_2}
\delta_{\gamma'_2\gamma'_3}\delta_{\gamma'_3\gamma'_4}\quad\quad\forall\;\gamma'\!\in\![N]^4.\]

\end{lem}

\begin{proof}
We prove the claim for \(d\!=\!1\), then it holds for all \(d\!\le\!N\).
We start with the partition {\vierpartwbwb}. Together with \(\paarpartwb\!\in\!\Pi\) we already know by Theorem \ref{thm:hom_from_PQS_to_QSymG} that 
\[\sum_{t'=1}^{N}\underbrace{v_{st'}v_{st'}^*v_{st'}v_{st'}^*}_{\le v_{st'}v_{st'}^*}=\mathds{1}=\sum_{t'=1}^N v_{st'}v_{st'}^*\quad,\quad\forall s\!\in\![N].\]

But this is only possible if \(v_{st'}v_{st'}^*v_{st'}v_{st'}^*\!=\! v_{st'}v_{st'}^*\) for all \(s,t'\in[N]\).
Hence, all generators \(v_{st'}\) are partial isometries.
Now, we also know \(\displaystyle\sum_{s}v_{st'}^*v_{st'}\!=\!\sum_{s}v_{st'}v_{st'}^*\!= \!\mathds{1}\) by Theorem \ref{thm:hom_from_PQS_to_QSymG}.
Hence \(v_{s_2t'}^*v_{s_1t'}\!=\!v_{s_2t'}v_{s_1t'}^*\!=\!0\) for \(s_2\!\neq\!s_1\), since projections summing up to one are mutually orthogonal.
This implies

%
\[\sum_{t'=1}^{N}v_{\gamma'_1t'}v_{\gamma'_2t'}^*v_{\gamma'_3t'}v_{\gamma'_4t'}^*=0\]
for all \(\gamma'\!\in\!T'_0\), which is the only relation in \(R^{Gr}_{\scalebox{0.8}{\vierpartwbwb}}(v_G)\) not already covered by the results in Theorem \ref{thm:hom_from_PQS_to_QSymG}.\vspace{11pt}

For the situation \(\vierpartwwbb\!\in\!\Pi\) we may assume that \(C(G)\) is represented faithfully on some Hilbert space \(H\).
As in the previous considerations, we get \(v_{st'}v_{st'}v_{st'}^*v_{st'}^*= v_{st'}v_{st'}^*\).
Thus, for any \(w\!\in\!H\)
\[
\norm{v_{s_1t'}^*v_{s_1t'}^*w}^2
=\langle v_{s_1t'}^*w,v_{s_1t'}^*w\rangle
=\langle\sum_{s_2=1}^N v_{s_2t'}v_{s_2t'}^*v_{s_1t'}^*w,v_{s_1t'}^*w\rangle
=\sum_{t_2=1}^N\norm{v_{s_2t'}^*v_{s_1t'}^*w}^2,
\]
which implies \(v_{s_2t'}^*v_{s_1t'}^*\!=\!v_{s_2t'}v_{s_1t'}\!=\!0\) for \(s_1\!\neq\!s_2\).
But then we also have
\[v_{s_2t'}v_{s_1t'}^*=\sum_{s=1}^{N}v_{s_2t'}v_{st'}v_{st'}^*v_{s_1t'}^*=0\]
and likewise \(v_{s_2t'}^*v_{s_1t'}=0\) for all \(t'\) and \(s_2\!\neq\!s_1\).
With this result it holds
\[\sum_{t'=1}^{N}v_{\gamma'_1t'}v_{\gamma'_2t'}v_{\gamma'_3t'}^*v_{\gamma'_4t'}^*=0\]
for all \(\gamma'\!\in\!T'_0\), which is the only relation in \(R'_{\scalebox{0.8}{\vierpartwwbb}}(v_G)\) to be proved.
\end{proof}

\begin{rem}\label{rem:four_block_implies_remaining_mcpp_relations}
Note that we were able to deduce \(v_{s_2t'}v_{s_1t'}^*=v_{s_2t'}^*v_{s_1t'}=0\) for \(s_2\!\neq\!s_1\) in the situation of both four-block partitions. As we can repeat the whole proof with \(v_{ij}\) replaced by \(v_{ji}\), see Theorem  \ref{thm:hom_from_PQS_to_QSymG}, we also have \(v_{t's_2}v_{t's_1}^*=v_{t's_2}^*v_{t's_1}=0\). So the non-diagonal entries of \(v_Gv_G^*\), \(v_G^*v_G\), \(\bar{v}_G\bar{v}_G^*\) and \(\bar{v}_G^*\bar{v}_G\) vanish. 
This is insofar a nontrivial result, as Theorem \ref{thm:hom_from_PQS_to_QSymG} together with the fact \(\{\paarpartwb,\paarpartbw,\baarpartwb,\baarpartbw\}\!\in\!\Pi\) only implies that the diagonals are equal to \(\mathds{1}\). At first side we did not know anything about the off-diagonals, i.e. it was unclear  if \(v_G\) and \(v_G^T\) are unitaries.

Note further, that if we additionally assume \(\paarpartww\!\in\!\Pi\), then also \(R_{\scalebox{0.7}{\(\paarpartww\)}}^{Gr}(v_G)\) are fulfilled, i.e. \(v_G\) is orthogonal, \(v_G^Tv_G=v_Gv_G^T=\mathds{1}\): The diagonals are \(\mathds{1}\) by Theorem \ref{thm:hom_from_PQS_to_QSymG}, and for the off-diagonals we can compute for \(k_1\!\neq\!k_2\)
\[\sum_{i=1}^{N}v_{ik_1}v_{ik_2}=\sum_{i=1}^{N}\sum_{k=1}^{N}v_{ik_1}v_{ik}^*v_{ik}^*v_{ik_2}=0\]
and likewise for the off-diagonals of \(v_Gv_G^T\).

\end{rem}

The next three lemmata are just preparing results. Recall (see \cite{banicaspeicherliberation,weber_LNM,weber_PMS}) that the easy quantum groups \(O_N^+\) and \(B_N^+\) can be associated to the following orthogonal \(N\!\times\!N\)-matries of generators and partitions:
\begin{align*}
O_N^+:&\quad u_{O_N^+}=(o_{ij})\quad;\quad\quad\Pi\!=\!\{\textnormal{non-crossing partitions with blocks of size 2}\}\\
B_N^+:&\quad u_{B_N^+}=(b_{ij})\quad;\quad\quad\Pi\!=\!\{\textnormal{non-crossing partitions with blocks of size 1 or 2}\}
\end{align*}

\begin{lem}\label{lem:dimension_of_subspaces_(O_2^+_and_B_3^+)}
\begin{itemize}
\item[(i)]The entries of the corepresentation matrix
\[u_{B_3^+}^{\otimes 2}=\sum_{m,n,s,t=1}^{3}b_{mn}b_{st}\otimes E_{mn}\otimes E_{st}\] 
linearly generate a vector space \(V\) of dimension 14.
\item[(ii)]The entries of the corepresentation matrix
\[u_{O_2^+}^{\otimes 2}=\sum_{m,n,s,t=1}^{2}o_{mn}o_{st}\otimes E_{mn}\otimes E_{st}\]
linearly generate a vector space \(W\) of dimension 10.
\end{itemize}
\end{lem}

\begin{proof}
One can deduce these results from the fusion rules for \(O_N^ +\) and \(B_N^+\) as described in \cite{freslonweber} and \cite{freslon}.
We use the notation introduced there and just sketch the main argument.
The vector space \(V\) is spanned by the (linearly independent) entries of the corepresentation matrices \(u_{\scalebox{0.5}{\idpartww\idpartww}}\), \(u_{\scalebox{0.5}{\idpartsingletonww\idpartww}}\) and \(u_{\scalebox{0.5}{\paarbaarpartwwww}}\).
The size of the matrix \(u_p\) (for a fixed projective partition \(p\)) is given by the rank of the projection
\[P_p:=T_p-\bigvee_{\substack{q\prec p\\[3pt]q\in \mathcal{C}_{B_N^+}}}T_q\]
and we have the formulas
\begin{align*}
\textnormal{rank}(T_p)&=N^{tb(p)}\\
\textnormal{rank}(P_p)&=\textnormal{rank}(T_p)- \sum_{\substack{q\prec p\\[3pt]q\in \mathcal{C}_{B_N^+}}}\textnormal{rank}(P_q)
\end{align*}
where \(\mathcal{C}_{B_N^+}\) is the category of partitions associated to \(B_N^+\) and \(\{q\;|\;q\prec p\}\) is the set of projective partitions over the same points as \(p\), but strictly smaller than \(p\).
Recursively, we therefore have with \(N\!=\!3\)
\begin{align*}
\dim(V)=&\left(\textnormal{rank}\left(P_{\scalebox{0.5}{\idpartww\idpartww}}\right)\right)^2+\left(\textnormal{rank}\left(P_{\scalebox{0.5}{\idpartsingletonww\idpartww}}\right)\right)^2+\left(\textnormal{rank}\left(P_{\scalebox{0.5}{\paarbaarpartwwww}}\right)\right)^2\\
=&\left(3^2
-\textnormal{rank}(P_{\scalebox{0.5}{\idpartsingletonww\idpartww}})
-\textnormal{rank}(P_{\scalebox{0.5}{\idpartww\idpartsingletonww}})
-\textnormal{rank}(P_{\scalebox{0.5}{\idpartsingletonww\idpartsingletonww}})
-\textnormal{rank}(P_{\scalebox{0.5}{\paarbaarpartwwww}})\right)^2\\
&+\left(3-\textnormal{rank}(P_{\scalebox{0.5}{\idpartsingletonww\idpartsingletonww}})\right)^2+1^2\\
=&\big(3^2-(3-1)-(3-1)-1-1\big)^2+(3-1)^2+1\\
=&14.
\end{align*}
 
 Analogously we have for \(O_2^+\) the result
 \begin{align*}
\dim(W)=&\left(\textnormal{\textnormal{rank}}\left(P_{\scalebox{0.5}{\idpartww\idpartww}}\right)\right)^2+\left(\textnormal{rank}\left(P_{\scalebox{0.5}{\paarbaarpartwwww}}\right)\right)^2\\
=&\left(2^2
-\textnormal{rank}(P_{\scalebox{0.5}{\paarbaarpartwwww}})\right)^2+1^2\\
=&10.
\end{align*}
\end{proof}

Note that for different easy quantum groups \(G\) the symbols \(P_p\) and \(u_p\), respectively, have different meanings as their definition depends on the considered category \(\mathcal{C}_{G}\) of the relevant quantum group.

\begin{lem}\label{lem:linear_independence_xy,yx_case_O_2^+}
Consider the matrix \(u_{O_2^ +}=(o_{ij})\) of the canonical generators of \(C(O_2^ +)\). Then \(o_{11}o_{21}\) and \(o_{21}o_{11}\) are linearly independent.
\end{lem}
\begin{proof}
Assume \(o_{11}o_{21}\) and \(o_{21}o_{11}\) are colinear. Switching columns or rows of \(u_{O_2^+}=(o_{ij})\) as well as taking the transpose defines isomorphisms of \(C(O_2^+)\), so we also have that the pairs \((o_{11}o_{12},o_{12}o_{11})\), \((o_{21}o_{22},o_{22}o_{21})\) and \((o_{12}o_{22},o_{22}o_{12})\) are each colinear. Together with the orthogonality of \(u_{O_2^+}\), i.e. \(\sum_{i=1}^{2}o_{ik_1}o_{ik_2}=\sum_{i=1}^{2}o_{k_1i}o_{k_2i}=\delta_{k_1k_2}\) for all \(k_1,k_2\!\in\!\{1,2\}\), this would imply that
\[\big(o_{11}o_{11}\,,\,o_{12}o_{12}\,,\,o_{11}o_{12}\,,\,o_{11}o_{21}\,,\,o_{11}o_{22}\,,\,o_{22}o_{11}\,,\,o_{12}o_{21}\,,\,o_{21}o_{12}\big)\]
is a generating system for the vector space \(W\) from Lemma \ref{lem:dimension_of_subspaces_(O_2^+_and_B_3^+)}, contradicting \(\dim(W)=10\).
\end{proof}

\begin{lem}\label{lem:linear_independence_xy,yx_case_B_3^+}
Consider the matrix \(u_{B_3^ +}=(b_{ij})\) of the canonical generators of \(C(B_3^ +)\).
Then \(b_{11}b_{21}\) and \(b_{21}b_{11}\) are linearly independent.
\end{lem}
\begin{proof}
The proof is similar to that of Lemma \ref{lem:linear_independence_xy,yx_case_O_2^+}.
Recall that \(b_{i3}\!=\!\mathds{1}-b_{i1}-b_{i2}\) and \(b_{3j}\!=\!\mathds{1}-b_{1j}-b_{2j}\), see \cite{banicaspeicherliberation,tarragoweberclassificationunitaryQGs,weber_LNM,weber_PMS}.
Thus, the elements \(b_{i3}\) and \(b_{3j}\) won't play a role in our proof.
Assume colinearity of \(b_{11}b_{21}\) and \(b_{21}b_{11}\).  Note again, that this implies further colinearities, namely of the pairs \((b_{11}b_{12},b_{12}b_{11})\), \((b_{21}b_{22},b_{22}b_{21})\) and \((b_{12}b_{22},b_{22}b_{12})\).
We observe now that the vector space \(V\) in Lemma \ref{lem:dimension_of_subspaces_(O_2^+_and_B_3^+)} is spanned by \(B:=\{\mathds{1},b_{11},\ldots,b _{22},b_{11}b_{11},\ldots, b_{22}b_{22}\}\), i.e. all 21 products of length at most two we can produce with the letters \(\{b_{11},b_{12},b_{21},b_{22}\}\).
Assuming the above, \(V\) is already linearly spanned by \(B\backslash\{b_{12}b_{11}, b_{21}b_{11}, b_{22}b_{21}, b_{22}b_{12}\}\).\newline
From \(\displaystyle\sum_{i}b_{i1}b_{i2}\!=\!0\) and the fact that each row and column in \(u_{B_3^+}\) sums up to \(\mathds{1}\) we can deduce
\[b_{21}b_{22}=-b_{11}b_{12}-b_{31}b_{32}=-b_{11}b_{12}-(\mathds{1}-b_{21}-b_{11})(\mathds{1}-b_{22}-b_{12})\]
and finally
\begin{equation}
2 b_{21}b_{22}=-2 b_{11}b_{12}-\mathds{1}-b_{11}b_{22}-b_{21}b_{12}+b_{11}+b_{12}+b_{21}+b_{22}
\end{equation}
Similarly, we can start with \(\displaystyle\sum_{i}b_{1i}b_{2i}=0\) or \(\displaystyle\sum_{i}b_{1i}b_{1i}=\sum_{i}b_{2i}b_{2i}=\mathds{1}\) to find
\begin{align}
2 b_{12}b_{22}&=-2 b_{11}b_{21}-\mathds{1}-b_{11}b_{22}-b_{12}b_{21}+b_{11}+b_{12}+b_{21}+b_{22}\\
2 b_{11}^2&=-2 b_{12}^2+2b_{11}+2b_{12}-2b_{11}b_{12}\\
2 b_{22}^2&=-2 b_{21}^2+2b_{22}+2b_{21}-2b_{21}b_{22}
\end{align}

This estimates the vector space dimension of \(V\) to at most 13, contradicting \(\textnormal{dim}(V)=14\).
\end{proof}

\begin{notation}\label{not:hom_from_PQS_to_M-version_of_smaller_BSQG}
One can show that for \(d\!\le\!M\!\le\!N\) we can always map a partition quantum space \(X_{N,d}(\Pi)\) on ``its shorter M-version'' \(X_{M,d}(\Pi)\):
\[\psi_1:C\big(X_{N,d}(\Pi)\big)\rightarrow C\big(X_{M,d}(\Pi)\big);\;x_{ij}\mapsto\begin{cases}x_{ij}'&\textnormal{, }j\!\le\!M\\0&\textnormal{, }j\!>\!M\end{cases}\]

Additionally, by Theorem \ref{thm:hom_from_PQS_into_BSQG_d-vector_case}, we always have a unital \(^*\)-homomorphism
\[\psi_2: C\big(X_{M,d}(\Pi)\big)\rightarrow C\big(G_M(\Pi')\big);x_{ij}\mapsto u_{ij}\]
whenever \(G_M(\Pi')\) is a subgroup of \(G_M(\Pi)\).
Composing \(\psi_1\) and \(\psi_2\) gives a unital \(^*\)-homomorphism \(\psi_{G_M(\Pi')}:C\big(X_{N,d}(\Pi)\big)\rightarrow C\big(G_M(\Pi')\big)\).
\end{notation}

\begin{lem}\label{lem:mcpp_relations_are_fulfilled_pairing_case}
Consider the situation of Notation \ref{not:fixing_notations} with \(d\!\le\!N\) arbitrary. If \(\Pi\) contains only non-crossing partitions with blocks of size two, then \(R^{Gr}_p(v_G)\) is fulfilled for every \(p\!\in\{\paarpartwb,\paarpartbw,\baarpartwb,\baarpartbw\}\). If \(\Pi\) contains \(\paarpartww\), then also \(R^{Gr}_{\scalebox{0.7}{\paarpartww}}(v_G)\) holds.
\end{lem}
\begin{proof}
We only prove the case \(p\!=\!\baarpartwb\). The other relations may be proved similarly. For the case {\paarpartww} replace every appearing \(v_{ij}^*\) and \(x_{ij}^*\) by \(v_{ij}\) and \(x_{ij}\), respectively.
If \(N\!=\!1\) or \(d\!\ge\!2\) then the result follows from Theorem \ref{thm:hom_from_PQS_to_QSymG}, so let \(d\!=\!1\!<\!N\).
By Theorem \ref{thm:hom_from_PQS_to_QSymG}, \(\baarpartwb\!\in\!\Pi\) implies that the diagonals of \(\bar{v}^*_G\bar{v}_G\) are equal to \(\mathds{1}\).
It remains to show, that the off-diagonals are zero.\newline
Applying \(\alpha\) to \(\displaystyle\sum_{s=1}^{N}x_{s1}x_{s1}^*=\mathds{1}\) yields \(\displaystyle\sum_{t_1,t_2=1}^{N}\sum_{s=1}^{N}v_{st_1}v_{st_2}^*\otimes x_{t_11}x^*_{t_21}=\mathds{1}\). As \(\Pi\) contains only non-crossing pairings, we have \(O_N^+\!\subseteq\!G_N(\Pi)\) so we have a mapping \(\psi_{O_2^+}\) as described in Notation \ref{not:hom_from_PQS_to_M-version_of_smaller_BSQG}.
Applying \(\mathds{1}\otimes\psi_{O_2^+}\) to this relation, we obtain
\[\sum_{t_1,t_2=1}^{2}\sum_{s=1}^{N}v_{st_1}v_{st_2}^*\otimes o_{t_11}o_{t_21}=\mathds{1}.\]
Using \(\sum_{s=1}^{N}v_{st_1}v_{st_1}^*=\mathds{1}=\sum_{t_1=1}^{2}o_{t_11}o_{t_11}\), this implies
\[\sum_{s=1}^{N}v_{s1}v_{s2}^*\otimes o_{11}o_{21}+\sum_{s=1}^{N}v_{s2}v_{s1}^*\otimes o_{21}o_{11}=0.\]
By linear independence of the right legs (see Lemma \ref{lem:linear_independence_xy,yx_case_O_2^+}) the left legs must be zero.

As the choice of \(x_{11}\) and \(x_{21}\) as those rows of \(x\) not being sent to zero by \(\psi_{O_2^+}\) was arbitrary we have for all \(\gamma_1\!\neq\gamma_2\) the result
\[\sum_{s=1}^{N}v_{s\gamma_1}v_{s\gamma_2}^*=0.\]
\end{proof}

\begin{lem}\label{lem:mcpp_relations_are_fulfilled_pairings_and_singleton-partition_case}
Consider the situation of Notation \ref{not:fixing_notations} with \(d\!\le\!N\) arbitrary. If \(\;\Pi\) only contains non-crossing partitions with blocks of size at most two and if every row and column of \(v_G\) sums up to \(\mathds{1}\), then \(R^{Gr}_p(v_G)\) is fulfilled for every \(p\!\in\!\{\paarpartwb,\paarpartbw,\baarpartwb,\baarpartbw\}\). If \(\Pi\) contains \(\paarpartww\), then also \(R^{Gr}_{\scalebox{0.7}{\paarpartww}}(v_G)\) holds.
\end{lem}
\begin{proof}
As in Lemma \ref{lem:mcpp_relations_are_fulfilled_pairing_case} we only care about the case \(1\!=d\!<\!N\) and we only consider \(p\!=\!\baarpartwb\). Again, the only thing left to prove is that the off-diagonals of \(\bar{v}_G^*\bar{v}_G\) vanish.

First assume \(N\!=\!2\).
We have by Theorem \ref{thm:hom_from_PQS_to_QSymG}
\[(\underbrace{v_{s1}+v_{s2}}_{=\mathds{1}})(\underbrace{v_{s1}^*+v_{s2}^*}_{=\mathds{1}})\!=\!\mathds{1}=v_{s1}v_{s1}^*+v_{s2}v_{s2}^*,\]
so \(v_{s1}v_{s2}^*=-v_{s2}v_{s1}^*\). Furthermore from
\[v_{11}+v_{12}=\mathds{1}=v_{11}+v_{21}\quad\textnormal{and}\quad v_{11}+v_{12}=\mathds{1}=v_{22}+v_{12}\]
we deduce \(v_{12}=v_{21}\) and \(v_{11}=v_{22}\).
Combining these relations yields for \(t_1\!\neq\!t_2\):
\[\sum_ {s=1}^{N}v_{st_1}v_{st_2}^*=v_{1t_1}v_{1t_2}^*+v_{2t_1}v_{2t_2}^*=v_{1t_1}v_{1t_2}^*+v_{1t_2}v_{1t_1}^*
=v_{1t_1}v_{1t_2}^*-v_{1t_1}v_{1t_2}^*=0.\]
For the rest of the proof, let \(N\!\ge\!3\).

\emph{Step 1}. We first prove
\begin{equation}\label{eqn:lem:{mcpp}_relations_are_fulfilled_pairings_and_singleton-partition_case_(2)}
\sum_{s=1}^{N}v_{st_1}v_{st_2}^*=-\sum_{s=1}^{N}v_{st_2}v_{st_1}^*\quad,\quad\forall t_1\!\neq\!t_2.
\end{equation}
Starting with the equation \(\displaystyle\sum_{s=1}^{N}x_{s1}x_{s1}^*=\mathds{1}\) we can apply \(\alpha\) to it and get
\[\sum_{s=1}^{N}\sum_{t_1,t_2=1}^{N}v_{st_1}v_{st_2}^*\otimes x_{t_11}x_{t_21}^*=\mathds{1}.\]
Using \(\sum_{s=1}^{N}v_{st_1}v_{st_1}^*\!=\!\mathds{1}\!=\!\sum_{t_1=1}^{N}x_{t_11}x_{t_11}^*\), we have
\begin{equation}\label{eqn:lem:{mcpp}_relations_are_fulfilled_pairings_and_singleton-partition_case_(1)}
\sum_{s=1}^{N}\sum_{t_1\neq t_2}v_{st_1}v_{st_2}^*\otimes x_{t_11}x_{t_21}^*=0.
\end{equation}
Consider now the two \(^*\)-homomorphisms \(\phi_1\) and \(\phi_2\) given by the mappings

\[\begin{pmatrix*}[r]
x_{11}\;\;\vspace{8pt}\\
x_{21}\;\;\vspace{8pt}\\
x_{31}\;\;\vspace{8pt}\\
\end{pmatrix*}
\overset{\phi_1}{\longmapsto} 
\begin{pmatrix*}[r]
-\frac{1}{3}\;\;\vspace{8pt}\\
\frac{2}{3}\;\;\vspace{8pt}\\
\frac{2}{3}\;\;\vspace{8pt}\\
\end{pmatrix*}
\quad\quad\textnormal{and}\quad\quad
\begin{pmatrix*}[r]
x_{11}\;\;\vspace{8pt}\\
x_{21}\;\;\vspace{8pt}\\
x_{31}\;\;\vspace{8pt}\\
\end{pmatrix*}
\overset{\phi_2}{\longmapsto} 
\begin{pmatrix*}[r]
\frac{2}{3}\;\;\vspace{8pt}\\
\frac{2}{3}\;\;\vspace{8pt}\\
-\frac{1}{3}\;\;\vspace{8pt}\\
\end{pmatrix*}.\vspace{6pt}\]

and all other \(x_{i1}\) sent to zero. To prove the existence of these maps, we observe the following:
By the conditions on \(\Pi\) we have \(B_N^+\subseteq G_N(\Pi)\).
Recall that the matrix \(u_{B_N^+}=(b_{ij})\) is orthogonal and  each row and column sums up to \(\mathds{1}\).
In the sense of Notation \ref{not:hom_from_PQS_to_M-version_of_smaller_BSQG} we can send \(x_{11},x_{21}\) and \(x_{31}\) by a map \(\psi_{B_3^+}\) to the first column of \(u_{B_3^+}\) and the rest to zero.
Finally, note that the complex vectors on the right appear as columns of matrices in \(B_3\!\subseteq\!B_3^+\) so in a second step we can map the \(\psi_{B_3^+}(x_{i1})\) to the complex numbers on the right sides.
Applying \(\mathds{1}\!\otimes\!\phi_1\) and \(\mathds{1}\!\otimes\!\phi_2\) to Equation \ref{eqn:lem:{mcpp}_relations_are_fulfilled_pairings_and_singleton-partition_case_(1)} leads to
\[\frac{4}{9}\sum_{s=1}^{N}(v_{s2}v_{s3}^*+v_{s3}v_{s2}^*)
-\frac{2}{9}\sum_{s=1}^{N}(v_{s1}v_{s2}^*+v_{s2}v_{s1}^*)
-\frac{2}{9}\sum_{s=1}^{N}(v_{s1}v_{s3}^*+v_{s3}v_{s1}^*)=0\]
and
\[-\frac{2}{9}\sum_{s=1}^{N}(v_{s2}v_{s3}^*+v_{s3}v_{s2}^*)
+\frac{4}{9}\sum_{s=1}^{N}(v_{s1}v_{s2}^*+v_{s2}v_{s1}^*)
-\frac{2}{9}\sum_{s=1}^{N}(v_{s1}v_{s3}^*+v_{s3}v_{s1}^*)=0,\]
which gives us in the end
\[\sum_{s=1}^{N}(v_{s1}v_{s2}^*+v_{s2}v_{s1}^*)=\sum_{s=1}^{N}(v_{s1}v_{s3}^*+v_{s3}v_{s1}^*).\]

As the choice of \((1,2,3)\) for the non-zero rows in the mappings \(\phi_1\) and \(\phi_2\) was arbitrary, we have this result for all pairwise different indices \((1,2,t)\), hence 
\[\displaystyle\sum_{s=1}^{N}v_{s1}v_{s2}^*+v_{s2}v_{s1}^*=\sum_{s=1}^{N}v_{s1}v_{st}^*+v_{st}v_{s1}^*\]
for all \(t\). In particular, since \(\sum_{t=1}^{N}v_{st}^*\!=\!\mathds{1}=\sum_{s=1}^{N}v_{s1}\):
\[(N-1)\sum_{s=1}^{N}v_{s1}v_{s2}^*+v_{s2}v_{s1}^*=\sum_{t=2}^{N}\sum_{s=1}^{N}v_{s1}v_{st}^*+v_{st}v_{s1}^*
=\mathds{1}-\sum_{s=1}^{N}v_{s1}v_{s1}^*+\mathds{1}-\sum_{s=1}^{N}v_{s1}v_{s1}^*=0\]
and because the indices \((1,2)\) were arbitrary this means for all \(t_1\!\neq\!t_2\)
\begin{equation*}
\sum_{s=1}^{N}v_{st_1}v_{st_2}^*=-\sum_{s=1}^{N}v_{st_2}v_{st_1}^*.\vspace{11pt}
\end{equation*}

\emph{Step 2:} We consider again Equation \ref{eqn:lem:{mcpp}_relations_are_fulfilled_pairings_and_singleton-partition_case_(1)}.
As in step 1 we can apply \(\psi_{B_3^+}\) to the second legs.
Using additionally Equation \ref{eqn:lem:{mcpp}_relations_are_fulfilled_pairings_and_singleton-partition_case_(2)} we find
\[\sum_{s=1}^{N}\sum_{t_1< t_2\le 3}v_{st_1}v_{st_2}^*\otimes \left(b_{t_11}b_{t_21}-b_{t_21}b_{t_11}\right)=0.\]
As \(b_{21}=\mathds {1}-b_{11}-b_{31}\) we easily see \(b_{11}b_{21}-b_{21}b_{11}=b_{21}b_{31}-b_{31}b_{21}=b_{31}b_{11}-b_{11}b_{31}\) and so we have
\[\left(\sum_{s=1}^{N}(v_{s1}v_{s2}^*+v_{s2}v_{s3}^*+v_{s3}v_{s1}^*)\right)\otimes \left(b_{11}b_{21}-b_{21}b_{11}\right)=0.\]
By Lemma \ref{lem:linear_independence_xy,yx_case_B_3^+}  we have \(b_{11}b_{21}\neq b_{21}b_{11}\), so the left leg of the tensor product must be zero.
The pairwise different indices \((1,2,3)\) were arbitrary, so using \((1,2,t)\) we have by Step 1
\begin{align*}
0&=\sum_{t=3}^{N}\sum_{s=1}^{N}\left(v_{s1}v_{s2}^*+v_{s2}v_{st}^*+v_{st}v_{s1}^*\right)&\\
&=\left((N-2)\sum_{s=1}^{N}v_{s1}v_{s2}^*\right)+\left(\mathds{1}-\sum_{s=1}^{N}v_{s2}(v_{s1}^*+v_{s2}^*)\right)+\left(\mathds{1}-\sum_{s=1}^{N}(v_{s1}+v_{s2})v_{s1}^*\right)&\\
&=(N-2)\sum_{s=1}^{N}v_{s1}v_{s2}^*-\sum_{s=1}^{N}v_{s2}v_{s1}^*-\sum_{s=1}^{N}v_{s1}v_{s1}^*\\
&=N\sum_{s=1}^{N}v_{s1}v_{s2}^*,\\
\end{align*}
giving us the desired relation for \(t_1\!=\!1\) and \(t_2\!=\!2\). As the choice of \((1,2)\) was arbitrary we proved the statement for general \(t_1\!\neq\!t_2\).
\end{proof}

Before continuing, we need the notion of a blockstable category of partitions:
\begin{defn}
We call a category \(\mathcal{C}\) of partitions \emph{blockstable}, if for every \(p\!\in\!\mathcal{C}\) and every block \(b\) of \(p\) we have \(b\!\in\!\mathcal{C}\).  In other words: By erasing all points (and lines) not belonging to \(b\), we obtain again a partition contained in \(\mathcal{C}\).
\end{defn}
We recall the classification of free easy quantum groups in the sense that the following sets \(\Pi\) generate all possible (and pairwise different) non-crossing categories of partitions (see  \cite[Thm. 7.1 and 7.2]{tarragoweberclassificationpartitions}).\vspace{-4pt}
\renewcommand*{\arraystretch}{1.8}
\begin{table}[H]\small
\begin{tabular}[l]{|m{19mm}|m{60mm}|m{32mm}|m{26mm}|}
\hline
Case & Elements in \(\Pi\)& Parameter range & Blockstable cases\\
\hhline{|=|=|=|=|}
\(\mathcal{O}_{\textnormal{loc}}\) &\(mcpp\)& -- & blockstable\\
\hline
\(\mathcal{H'}_{\textnormal{loc}}\) &\(\vierpartwbwb,mcpp\)& -- & blockstable\\
\hline
\(\mathcal{H}_{\textnormal{loc}}(k,l)\) & \(b_k,b_l\otimes\bar{b}_l,\vierpartwwbb,mcpp\)& \(k,l\!\in\!\N_0\backslash\{1,2\},l|k\) & \(k\!=\!l\)\\
\hline
\(\mathcal{S}_{\textnormal{loc}}(k,l)\) &\({\singletonw}^{\otimes k},\positionerl, \vierpartwbwb, \singletonw\otimes\!\!\singletonb,mcpp\)& \(k,l\!\in\!\N_0\backslash\{1\},l|k\) & not blockstable\\
\hline
\(\mathcal{B}_{\textnormal{loc}}(k,l)\) &\(\singletonw^{\otimes k},\positionerl,\singletonw\otimes\!\!\singletonb,mcpp\)& \(k,l\!\in\!\N_0,l|k\) & \(k\!=\!l\!=\!1\)\\
\hline
\(\mathcal{B'}_{\textnormal{loc}}(k,l,0)\) &\(\singletonw^{\otimes k},\positionerl,\positionerwbwb,\singletonw\otimes\!\!\singletonb,mcpp\)& \(k,l\!\in\!\N_0\backslash\{1\},l|k\) & not blockstable\\
\hline
\(\mathcal{B'}_{\textnormal{loc}}(k,l,\frac{l}{2})\) &\(\singletonw^{\otimes k},\positionerl,\positionerrpluseins,\positionerwwbb,\vspace{4pt}\newline\singletonw\otimes\!\!\singletonb,mcpp\)& \(k\!\in\!\N_0\backslash\{1\}\),\vspace{3pt}\newline \(\,l\!\in\!2\N_0\backslash\{0,2\}\),\vspace{3pt}\newline
\(\,l|k,\,r\!=\!\frac{l}{2}\) & not blockstable\\
\hline
\(\mathcal{O}_{\textnormal{glob}}(k)\) &\(\paarpartww^{\scalebox{0.8}{\raisebox{-3pt}{\(\small\otimes \frac{k}{2}\)}}},\paarpartww\otimes\!\paarpartbb,mcpp\)& \(k\!\in\! 2\N_0\) & \(k\!=\!2\)\\
\hline
\(\mathcal{H}_{\textnormal{glob}}(k)\) &\(b_k,\vierpartwbwb,\paarpartww\otimes\!\paarpartbb,mcpp\)&\(k\!\in\!2\N_0\) & \(k \!=\!2\)\\
\hline
\(\mathcal{S}_{\textnormal{glob}}(k)\) &\(\singletonw^{\otimes k},\!\vierpartwbwb, \!\singletonw\otimes\!\!\singletonb, \paarpartww\otimes\!\paarpartbb,mcpp\)&\(k\!\in\!\N_0\) &\(k\!=\!1\)\\
\hline
\(\mathcal{B}_{\textnormal{glob}}(k)\)&\(\singletonw^{\otimes k}, \singletonw\otimes\!\!\singletonb, \paarpartww\otimes\!\paarpartbb,mcpp\)& \(k\!\in\!2\N_0\) &not blockstable\\
\hline
\(\mathcal{B}'_{\textnormal{glob}}(k)\)&\(\singletonw^{\otimes k}, \positionerwwbb, \singletonw\otimes\!\!\singletonb, \paarpartww\otimes\!\paarpartbb,mcpp\) &\(k\!\in\!\N_0\) &\(k\!=\!1\)\\
\hline
\end{tabular}
\caption{Here, \(b_k\)/\(\bar{b}_k\) is the one-block partition in \(P(0,k)\) with only white/black points and with \(mcpp\) we denote the four mixed-coloured pair partitions \(\{\paarpartwb,\paarpartbw,\baarpartwb,\baarpartbw\}\).}
\label{table:classification_free_case}
\end{table}\vspace{0pt}

\renewcommand*{\arraystretch}{1}

\begin{thm}\label{thm:free_case_main_result}
Let \(N\!\in\!\N\backslash\{1\}\) and fix any of the sets \(\Pi\) presented in Table \ref{table:classification_free_case}. In the case \(d\!=\!2\), \(G_N(\Pi)\) is the quantum symmetry group of \(X_{N,2}(\Pi)\).\newline
If the category \(\langle\Pi\rangle\) is blockstable, or if \(N\!=\!1\), then this results even holds for \(d\!=\!1\), i.e. \(X_{N,1}(\Pi)\).

%
\end{thm} 

\begin{proof}
We consider again the situation as in Notation \ref{not:fixing_notations}.
By Theorem \ref{thm:existence_of_alpha_and_beta} we know \(G_N(\Pi)\!\subseteq\!G\), so we only need to show that \(R^{Gr}_p(v_G)\) is fulfilled for all \(p\!\in\!\Pi\).
As the case \(N\!=\!d\!=\!1\) is by Corollary \ref{cor:X_is_BSQS_if_d=N}, we assume \(N\!\ge\!2\).

We have that the relations coming from the partitions \(\{\paarpartwb,\paarpartbw,\baarpartwb,\baarpartbw\}\) are fulfilled. For \(d\!=\!2\) this is Theorem \ref{thm:hom_from_PQS_to_QSymG} and for \(d\!=\!1\) see Remark \ref{rem:four_block_implies_remaining_mcpp_relations} and Lemmata \ref{lem:mcpp_relations_are_fulfilled_pairing_case} and \ref{lem:mcpp_relations_are_fulfilled_pairings_and_singleton-partition_case}. The same holds for {\paarpartww}, where required.
Most of the remaining parts of the proof are by Theorem \ref{thm:hom_from_PQS_to_QSymG} and Lemma \ref{lem:relations_of_four-block} but in virtue of Remark \ref{rem:proof-plan_for_d<N} we often have to perform some algebraic operations to see that the desired relations are really fulfilled for all relevant multi indices \(\gamma\) and \(\gamma'\). We prove two cases. The other ones are handled with similar arguments.\newline


\emph{Case \(\mathcal{S}_{\textnormal{loc}}(k,l)\):} By the arguments above the relations due to \(mcpp\) and \(\vierpartwbwb\) are fulfilled.
As \(d\!=\!2\) also the relations \(R^{Gr}_{\scalebox{0.7}{\(\singletonw\!\otimes\!\!\singletonb\)}}(v_G)\) are guaranteed by Theorem \ref{thm:hom_from_PQS_to_QSymG}. In the case \(k\!=\!l\!=\!0\) this is everything to be proved.
From the fact that \(v_G\) and \(v_G^T\) are unitaries and Lemma \ref{lem:relations_on_generators_imply_relations_for_category} we deduce that also \(R^{Gr}_{\scalebox{0.7}{\(\singletonb\!\otimes\!\!\singletonw\)}}(v_G)\) are fulfilled. But this now guarantees that each row and column of \(v_G\)  sums up to the same (unitary) element.
Using this result we can consider now the relations \(R^{Gr}_{\scalebox{0.7}{\(\singletonw^{\otimes k}\)}}(v_G)\) reading as
\[\sum_{t_1,\ldots,t_k}v_{\gamma'_1t_1}\cdots v_{\gamma'_kt_k}=\mathds{1}\]
which are now proved to be true not only for \(\gamma'\) with at most two different entries (see Theorem \ref{thm:hom_from_PQS_to_QSymG}) but for all \(\gamma\!\in\![N]^k\).
The same argument secures all the quantum group relations associated to \(p\!=\!\positionerl\), which read as
\begin{equation*}\label{eqn:relation_for_v_ij_positionerl}
\sum_{t'_1,\ldots,t'_{2l+1}}\big(v_{\gamma_1t'_1}\cdots v_{\gamma'_lt'_l}\big)v_{\gamma'_{l+1}t'_{l+1}}\big(v_{\gamma'_{l+2}t'_{l+2}}^*\cdots v_{\gamma'_{2l+1}t'_{2l+1}}^*\big)v_{\gamma'_{2l+2}t'_{l+1}}^*=\delta_{\gamma'_{l+1},\gamma'_{2l+2}}.
\end{equation*}
At first site these are true only if \(\gamma'=(\gamma'_1,\ldots,\gamma'_{2l+2})\) has at most two different entries,
but by the arguments from above we can replace all entries \(\gamma'_1,\ldots,\gamma_l,\gamma'_{l+2},\ldots,\gamma'_{2l+1}\) by \(\gamma'_{2l+2}\), proving the claim.\newline
\emph{Case \(\mathcal{O}_{\textnormal{glob}}(k)\):} For \(k\!=\!2\) we only need to prove \(R_p^{Gr}(v_G)\) for \(p\!\in\!\Pi':=\{\paarpartww,mcpp\}\) which is Lemma \ref{lem:mcpp_relations_are_fulfilled_pairing_case}.
For \(k\!\in\!2\N\backslash\{2\}\) this Lemma only guarantees the relations due to \(p\!\in\!\{mcpp\}\).
We start with the partition \(p\!=\!\paarpartww\otimes\paarpartbb\). The corresponding quantum group relations read as
\[\left(\sum_{t'_1}v_{\gamma'_1t'_1}v_{\gamma'_2t'_1}\right)
\left(\sum_{t'_2}v_{\gamma'_3t'_2}^*v_{\gamma'_4t'_2}^*\right)
=\delta_{\gamma'_1\gamma'_2}\delta_{\gamma'_3\gamma'_4}\]
and Theorem \ref{thm:hom_from_PQS_to_QSymG} only guarantees this result for \(\gamma'\) with at most two different entries. But choosing \(\gamma'_1=\gamma'_4\!\neq\!\gamma'_2\!=\!\gamma'_3\) shows
\[\left(\sum_{t'_1}v_{\gamma'_1t_1}v_{\gamma'_2t_1}\right)\left(\sum_{t'_2}v_{\gamma'_2t'_2}^*v_{\gamma'_1t'_2}^*\right)
=\left(\sum_{t'_1}v_{\gamma'_1t_1}v_{\gamma'_2t_1}\right)\left(\sum_{t'_1}v_{\gamma'_1t_1}v_{\gamma'_2t_1}\right)^*=0,\]
so 
\[\sum_{t'_1}v_{\gamma'_1t_1}v_{\gamma'_2t_1}=0\quad\quad\forall \gamma'_1\!\neq\!\gamma'_2.\]
This proves \(R^{Gr}_{\scalebox{0.7}{\(\paarpartww\otimes\paarpartbb\)}}(v_G)\).
Together with the fact that \(v_G\) and \(v_G^T\) are unitaries we also have that the quantum group relations for \(\paarpartbb\otimes\paarpartww\) are fulfilled, so the sums \(\sum_{t_1}v_{\gamma'_1t'_1}v_{\gamma'_1t'_1}\) are unitaries, so invertible.
Therefore, \(R^{Gr}_{\scalebox{0.7}{\(\paarpartww\otimes\paarpartbb\)}}(v_G)\) in particular says that \(\sum_{t'_1}v_{\gamma'_1t'_1}v_{\gamma'_1t'_1}\) is independent of \(\gamma'_1\!\in\![N]\).
We finally use all these results in the situation of \(p\!=\!\paarpartww^{\otimes k}\) to show that the corresponding relations,
\[\left(\sum_{t'_1}v_{\gamma'_1t'_1}v_{\gamma'_2t'_1}\right)\cdots
\left(\sum_{t'_k}v_{\gamma'_{2k-1}t'_k}v_{\gamma'_{2k}t'_{k}}\right)
=\delta_{\gamma'_1\gamma'_2}\cdots\delta_{\gamma'_{2k-1}\gamma'_{2k}},\]
are true for all \(\gamma'\!\in\![N]^{2k}\), as we can now make the replacement
\[(\gamma'_{2m+1},\gamma'_{2m+2})\mapsto\begin{cases}(1,1)&,\gamma'_{2m+1}=\gamma'_{2m+2}\\(1,2)&,\gamma'_{2m+1}\neq\gamma'_{2m+1}\end{cases}.\]
\end{proof}

\begin{rem}\label{rem:group_case_main_result}
Adding the crossing partition \scalebox{0.7}{\crosspartwwww} to the sets \(\Pi\) in Table \ref{table:classification_free_case} produces all categories for all unitary easy groups, see \cite{tarragoweberclassificationunitaryQGs}.
It obviously guarantees commutativity of the \(x_{ij}\)'s and for \(d\!=\!2 \) we have \(R^{Gr}_{\scalebox{0.5}{\crosspartwwww}}(v_G)\) fulfilled by Theorem \ref{thm:hom_from_PQS_to_QSymG}.
So the (quantum) symmetry groups of these partition (quantum) spaces are given by the corresponding easy groups.
Note that for \(d\!=\!2\) we can directly deduce from \(\{\paarpartwb,\paarpartbw,\baarpartwb,\baarpartbw\}\!\subseteq\!\Pi\) that \(v_G\) and \(\bar{v}_G\) are unitaries, so we do not need to use Lemmata \ref{lem:mcpp_relations_are_fulfilled_pairing_case} and \ref{lem:mcpp_relations_are_fulfilled_pairings_and_singleton-partition_case}. We finally remark that it is unclear, if \(d\!=\!1\) works in the blockstable cases.
\end{rem}

\section{Open questions and further remarks}\label{sec:open_questions}
\begin{quest}\label{quest:when_does_QSymG_only_depend_on_category?}
Are there situations or conditions (apart from \(d\!=\!N\)) such that the quantum symmetry group (or even the PQS) only depends on \(\langle\Pi\rangle\) and not \(\Pi\) itself?
\end{quest}
Regarding Corollary \ref{cor:QSymG_only_depends_on_category_if_d=N}, there is a simple counterexample for the analogous statement with \(d\!=\!1\): The free hyperoctahedral group \(H_N^+\) corresponds to the case \( \mathcal{H}_{\textnormal{loc}}(2,2)\), i.e. \(\Pi=\{\paarpartww,\vierpartwwbb,mcpp\}\), see Table \ref{table:classification_free_case}. The category of partitions \(\langle\Pi\rangle\) is also generated by \(\Pi':=\{\paarpartww,\scalebox{0.8}{\vierpartrotwbwb},mcpp\}\) but obviously \(R^{Sp}_{\scalebox{0.7}{\vierpartrotwbwb}}(x)\) is just the trivial relation. We thus have \(X_{N,1}(\Pi)\!\neq\!X_{N,1}(\Pi')\) and the quantum symmetry group of \(X_{N,1}(\Pi')\) is \(O_N^+\) (i.e. case \(\mathcal{O}_{\textnormal{glob}}(2)\)), whereas the one of \(X_{N,1}(\Pi)\) is \(H_N^+\).

\begin{quest}
Can we produce results similar to Theorem \ref{thm:free_case_main_result} (free case) or Remark \ref{rem:group_case_main_result} (group case) for other classes of partitions/easy quantum groups?
\end{quest}

\begin{quest}
 Is there  a way to read off from \(\Pi\) the minimal \(d\) such that \(G_N(\Pi)\) is the quantum symmetry group of \(X_{N,d}(\Pi)\)? In the situation of Theorem \ref{thm:free_case_main_result}, is \(d\!=\!1\) equivalent to \(\langle\Pi\rangle\) being blockstable?
\end{quest}
Due to Theorem \ref{thm:free_case_main_result}, we have \(d\!\!=\!\!1\) in the free blockstable cases, at least for the choices of \(\Pi\) presented in Table \ref{table:classification_free_case}.
But the counterexample after Question \ref{quest:when_does_QSymG_only_depend_on_category?} already shows that there are other choices for \(\Pi\), even in the non-crossing situation, where this is not true.
Another example from the commutative case is the partition set \(\Pi=\{\scalebox{0.7}{\crosspartwwww},\paarpartww,\vierpartwwbb,mcpp\}\) corresponding to the hyperoctahedral group \(H_N\).
For \(d\!\!=\!\!1\) we have \(X_{N,1}(\Pi)\!=\!X_{N,1}\big(\Pi\backslash\{\scalebox{0.6}{\crosspart}\}\big)\) as commutativity already follows from \(R^{Sp}_{\scalebox{0.7}{\(\vierpartwwbb\)}}(v_G)\)\vspace{5pt}, see Lemma \ref{lem:relations_of_four-block}, so the corresponding quantum symmetry group is \(H_N^+\) by Theorem \ref{thm:free_case_main_result}.

In the cases presented in Table \ref{table:classification_free_case} we have some sets \(\Pi\) where the quantum symmetry group of \(X_{N,1}(\Pi)\) is not given by \(G_N(\Pi)\), supporting our conjecture, that the case \(d\!=\!1\) is linked to blockstability.
Consider for example \(\mathcal{H}_{\textnormal{loc}}(k,l)\) with \(k\!\neq\!l\) and let \(\Pi(k,l)\) be the corresponding  set of partitions from Table \ref{table:classification_free_case}.
We have \(X_{N,1}(\Pi(k,l))=X_{N,1}(\Pi(k,k))\) as the quantum space relations \(R^{Sp}_{b_l\otimes\bar{b}_l}(x)\) are redundant.
Therefore the quantum symmetry group of \(X_{N,1}\big(\Pi(k,l)\big)\) is \(G_N\big(\Pi(k,k)\big)\) which is in general bigger than \(G_N\big(\Pi(k,l)\big)\).
Similar results hold in the cases \(\mathcal{S}_{\textnormal{loc}}(0,0)\), \(\mathcal{H}_{\textnormal{glob}}(k)\) for \(k\!\in\!2\N+4\) and \(\mathcal{S}_{\textnormal{glob}}(0)\)
, where respectively \(R^{Sp}_{\scalebox{0.7}{\(\singletonw\!\otimes\!\!\singletonb\)}}(x)\), \(R^{Sp}_{\scalebox{0.7}{\(\paarpartww\!\otimes\!\!\paarpartbb\)}}(x)\)\vspace{4pt} and again \(R^{Sp}_{\scalebox{0.7}{\(\singletonw\!\otimes\!\!\singletonb\)}}(x)\) are redundant.

On the other hand, though, we cannot guarantee that \(d\!=\!1\) fails in all non-blockstable cases.
Our standard method to deduce relations for the \(v_{ij}\) was to start with a quantum space relation \(R_p^{Sp}(x)\), apply \(\alpha\) or \(\beta\) to it and finally \(\mathds{1}\!\otimes\!\textnormal{ev}_{\sigma}\).
But of course by this procedure we might have lost some information as \(\textnormal{ev}_G\) is far from being an isomorphism. In principle we would have to stay inside \(X_{N,d}(\Pi)\) or at least \(G_N(\Pi)\).
In \(G_N(\Pi)\) we could deduce many (in)dependencies by the fusion rules established in \cite{freslonweber} and \cite{freslon} as done in Lemma \ref{lem:dimension_of_subspaces_(O_2^+_and_B_3^+)} and the ones following thereafter.
Hence, although we expect that for non-blockstable categories we always need \(d\!\ge\!2\) in order to reconstruct \(G_N(\Pi)\) as the quantum symmetry group of \(X_{N,d}(\Pi)\), we have to leave this question open.
\bibliographystyle{plain}
\bibliography{pqs}

\section{Appendix}
\begin{notation*}[see Notation \ref{not:partitioning_of_labelings}]
Given \(p\!\in\!\mathcal{P}(k,l)\) and \(N\!\in\!\N\) we write\vspace{-2pt}
\[[N]^k=T_0\,\dot{\cup}\,T_1\,\dot{\cup}\,\ldots\,\dot{\cup}\,T_r\quad,\quad[N]^l=T'_0\,\dot{\cup}\,T'_1\,\dot{\cup}\,\ldots\,\dot{\cup}\,T'_r,\vspace{-6pt}\]
such that 
\begin{compactitem}
\item [(i)] \(r=N^{tb(p)}\), where \(tb(p)\) denotes the number of through-blocks of \(p\),
\item [(ii)]\(T_0\) and \(T'_0\) are the invalid labelings of the upper (respectively lower) row,
\item [(iii)]for every \(1\!\le\!i\!\le\!r\) every labeling \((t,t')\!\in\!T_i\!\times\!T'_i\) is valid,
\item[(iv)] for every \(1\!\le\!i\!\le\!r\) the sets \(T_i\) and \(T'_i\) are non-empty,
\item [(v)]if \((t,t')\!\in\![N]^k\!\times\![N]^l\) is a valid labeling, then \((t,t')\!\in\!T_i\!\times\!T'_i\)  for some \(1\!\le\!i\!\le\!r\),
\item [(vi)]for every \(1\!\le\!i\!\le\!r\) and \((t,t'),(s,s')\!\in\!T_i\!\times\!T'_i\) we have that \((t,t')\) labels the through-blocks of \(p\) the same way as \((s,s')\) does.
\end{compactitem} 
\end{notation*}
\newpage
Here is a summary of all relations associated to partitions on one page.
\begin{defn*}[see Definition \ref{defn:quantum_group_relations_using_T_i's}]
Let \(N\!\in\!\N\), \(u\!:=\!(u_{ij})\) an \(N\!\times\!N\)-matrix of generators and \(p\!\in\!\mathcal{P}(\omega,\omega')\!\subseteq\!\mathcal{P}(k,l)\) be a partition. The relations \(R^{Gr}_p(u)\) are:
\begin{itemize}
\item[(i)]
\(\displaystyle\sum_{t\in T_i} u_{t_1\gamma_1}^{\omega_1}\cdots u_{t_k\gamma_k}^{\omega_k}
=\sum_{t'\in T'_j} u_{\gamma'_1t'_1}^{\omega'_1}\cdots u_{\gamma'_lt'_l}^{\omega'_l}\quad,\quad 1\!\le \!i,j\!\le\!r\), \(\gamma\!\in\!T_{j}\textnormal{ and }\gamma'\!\in\!T'_{i}.\)

\item[(ii)]
\(\displaystyle\sum_{t\in T_i}u_{t_1\gamma_1}^{\omega_1}\cdots u_{t_k\gamma_k}^{\omega_k}=0\quad,\quad 1\!\le\!i\!\le\!r\textnormal{ and }\gamma\!\in\!T_0.\)

\item[(iii)]
\(\displaystyle\sum_{t'\in T'_j}u_{\gamma'_1t'_1}^{\omega'_1}\cdots u_{\gamma'_lt'_l}^{\omega'_l}=0\quad,\quad 1\!\le\!j\!\le\!r\textnormal{ and }\gamma'\!\in\!T'_0.\)
\end{itemize}
\end{defn*}

\begin{defn*}[see Definition \ref{defn:first_d-many_columns-relations}]
The relations \(R^ {Sp}_p(u_G)\) are:
\begin{itemize}
\item[(i)]\(\displaystyle\sum_{t\in T_i}u_{t_1\gamma_1}^{\omega_1}\dots u_{t_k\gamma_k}^{\omega_k}=\sum_{t'\in T'_i}u_{t'_1\gamma'_1}^{\omega'_1}\dots u_{t'_l\gamma'_l}^{\omega'_l}\quad,\quad1\!\le\!i,j\!\le\!r,\gamma\!\in\!T_j\textnormal { and }\gamma'\!\in\!T'_j\).
\item[(ii)]\(\displaystyle\sum_{t\in T_i}u_{t_1\gamma_1}^{\omega_1}\dots u_{t_k\gamma_k}^{\omega_k}=0\quad,\quad1\!\le\!i\!\le\!r\textnormal { and }\gamma\!\in\!T_0\).
\item[(iii)]\(\displaystyle\sum_{t'\in T'_i}u_{t'_1\gamma'_1}^{\omega'_1}\dots u_{t'_l\gamma'_l}^{\omega'_l}=0\quad,\quad1\!\le\!i\!\le\!r\textnormal { and }\gamma'\!\in\!T'_0\).
\end{itemize}
\end{defn*}

\begin{lem*}[see Lemma \ref{lem:quantum_group_relations_imply_quantum_space_relations} and Remark \ref{rem:quantum_space_relations_always_fulfilled_for_BSQGs}]
It holds
\begin{itemize}
\item [(1)]\(R_p^{Gr}(u_G),\;R_{pp^*}^{Gr}(u_G),\;R_{p^*}^{Gr}(u_G)\quad\Rightarrow\quad R_p^{Sp}(u_G)\),
\item [(2)]\(R_p^{Sp}(u_G),\;R_p^{Sp}(u_G^T)\quad\Rightarrow\quad R_p^{Gr}(u_G)\).
\item[(3)]If \(G\!=\!G_N(\Pi)\), then: \(R_p^{Gr}\big(u_{G_N(\Pi)}\big)\Leftrightarrow R_p^{Sp}\big(u_{G_N(\Pi)}\big),R_p^{Sp}\big(u_{G_N(\Pi)}^T\big)\).
\end{itemize}
\end{lem*}

\begin{defn*}[see Definition \ref{defn:quantum_space_relations_for_the_x_ij}]
Let \(d,N\!\in\!\N\) with \(d\!\le\!N\) and \((x_{ij})_{1\le i\le N,1\le j\le d}\) a tupel of vectors of generators \(x_{ij}\).
The relations \(R^{Sp}_p(x)\) are:
\begin{itemize}
\item[(i)]\(\displaystyle\sum_{t\in T_i}x_{t_1\gamma_1}^{\omega_1}\dots x_{t_k\gamma_k}^{\omega_k}=\sum_{t'\in T'_i}x_{t'_1\gamma'_1}^{\omega'_1}\dots x_{t'_l\gamma'_l}^{\omega'_l}\quad,\quad 1\!\le\!i,j\!\le\!r,\; \gamma\!\in\!T_j\!\cap\![d]^k,\;\gamma'\!\in\!T'_j\!\cap\![d]^l\).\vspace{4pt}
\item[(ii)]\(\displaystyle\sum_{t\in T_i}x_{t_1\gamma_1}^{\omega_1}\dots x_{t_k\gamma_k}^{\omega_k}=0\quad,\quad 1\!\le\!i\!\le\!r,\;\gamma\!\in\!T_0\!\cap\![d]^k\).\vspace{4pt}
\item[(iii)]\(\displaystyle\sum_{t'\in T'_i}x_{t'_1\gamma'_1}^{\omega'_1}\dots x_{t'_l\gamma'_l}^{\omega'_l}=0\quad,\quad 1\!\le\!i\!\le\!r,\;\gamma'\!\in\!T'_0\!\cap\![d]^l\).\vspace{8pt}
\end{itemize}
\end{defn*}

\end{document}